\documentclass[11pt,reqno]{amsart}

\usepackage[T1]{fontenc}
\usepackage[utf8]{inputenc}
\usepackage{lmodern}
\usepackage{microtype}

\usepackage{amsmath,amssymb,amsthm,mathtools}
\usepackage{mathrsfs}
\usepackage{bm}
\usepackage{bbm}

\usepackage{geometry}
\usepackage{enumitem}
\usepackage{graphicx}
\usepackage{xcolor}
\usepackage{tikz}
\usepackage{cancel}

\usepackage{physics}

\usepackage{hyperref}
\usepackage[nameinlink,capitalise]{cleveref}

\title[Quantitative Stability for Fractional Yamabe minimizers]
{Quantitative Stability for Fractional Yamabe minimizers}

\author[A. S. Aiken]{A. Sophie Aiken}

\address{(A. S. Aiken) Department of Mathematics \& Statistics,
Metropolitan State University of Denver, Denver, Colorado, 80204,
USA
}

\email{auaiken@msudenver.edu}

\author[B. Borquez]{Benjamin Borquez}

\address{(B. Borquez) Department of Mathematics,
University of California, Santa Cruz, 4111 McHenry, Santa Cruz, CA 95064,
USA
}

\email{bborquez@ucsc.edu}

\subjclass[2020]{Primary 53C21; Secondary 53C18, 35R11, 35J60.}

\keywords{Fractional Yamabe problem, Quantitative stability, Fractional conformal Laplacian, Lyapunov--Schmidt reduction.}

\date{}

\hypersetup{
    colorlinks=true,
    linkcolor=blue,
    citecolor=blue,
    urlcolor=blue,
    pdftitle={Quantitative Stability for Fractional Yamabe minimizers},
    pdfauthor={A. Sophie Aiken and Benjamin Borquez}
}

\theoremstyle{plain}
\newtheorem{theorem}{Theorem}[section]
\newtheorem{lemma}[theorem]{Lemma}
\newtheorem{proposition}[theorem]{Proposition}
\newtheorem{corollary}[theorem]{Corollary}
\newtheorem{theoremletter}{Theorem}

\theoremstyle{definition}
\newtheorem{definition}[theorem]{Definition}

\newtheorem{remark}[theorem]{Remark}

\numberwithin{equation}{section}

\begin{document}
\begin{abstract}
We study quantitative stability for the fractional Yamabe functional introduced by González and Qing. Its minimizers correspond to conformal metrics with constant fractional curvature. We prove that the deficit of the fractional Yamabe functional controls, with a suitable power, the distance to the set of minimizers. Our approach is based on a Lyapunov--Schmidt reduction for the nonlocal functional, together with the extension characterization of the fractional conformal Laplacian, which provides the regularity theory needed for the local analysis. As a consequence, we also obtain an analogous quantitative stability estimate for the associated extension functional.
\end{abstract}
\maketitle

\bigskip

\setcounter{tocdepth}{1}
\tableofcontents

\section{Introduction}
One of the central themes in conformal geometry is the study of geometric quantities that remain invariant under conformal changes of the metric. This naturally leads to conformally covariant differential operators, whose transformation laws allow one to formulate geometric PDEs that are invariant under conformal deformations.

The first fundamental example is the conformal Laplacian, which naturally arises in the classical Yamabe problem. More generally, Graham, Jenne, Mason and Sparling introduced in \cite{GJMS} a hierarchy of local conformally covariant differential operators, nowadays known as the GJMS operators. These operators generalize the conformal Laplacian to higher orders, have principal part given by powers of the Laplacian, and include among their most notable examples the conformal Laplacian and the Paneitz operator. A natural question is whether this hierarchy admits an extension to non-integer orders. Such an extension requires introducing the geometric setting of asymptotically hyperbolic manifolds.

Let $(X^{n+1},g^+)$ be a smooth Riemannian manifold with smooth boundary $M^n$. We say that $(X^{n+1},g^+)$ is \emph{conformally compact} if there exists a smooth \emph{boundary defining function} $\rho$ satisfying
\[
\rho\geq0 \text{ on }X^{n+1},\qquad
\rho=0 \text{ on }M^n,\qquad
d\rho\neq0 \text{ on }M^n,
\]
such that the compactified metric $\overline g=\rho^2g^+$ extends smoothly to $\overline X^{n+1}$ and $(\overline X^{n+1},\overline g)$ is compact. Let $\hat h=\overline g|_{TM}$ denote the metric induced on the boundary. Since different choices of defining function produce conformally equivalent boundary metrics, they determine a conformal class $[\hat h]$ on $M^n$, called the \emph{conformal infinity} of $(X^{n+1},g^+)$. Finally, $(X^{n+1},g^+)$ is said to be \emph{asymptotically hyperbolic} if it is conformally compact and its sectional curvature approaches $-1$ as $\rho\to0$.

Following the study of the spectrum of the Laplacian on asymptotically hyperbolic manifolds \cite{MazzeoMelrose,Mazzeo,Mazzeo2}, the scattering theory developed in \cite{MazzeoMelrose,GrahamZworski} provides a natural way to associate boundary operators to the geometry of $(X^{n+1},g^+)$. More precisely, given $f\in C^\infty(M)$ and $s\in\mathbb C$ with $\operatorname{Re}(s)>\frac{n}{2}$ such that $s(n-s)$ is not an $L^2$-eigenvalue of $-\Delta_{g^+}$, the generalized eigenvalue problem
\begin{align*}
    -\Delta_{g^+}u-s(n-s)u=0
    \qquad \text{in }(X^{n+1},g^+)
\end{align*}
admits a solution with asymptotic expansion
\begin{align*}
    u=\rho^{n-s}F+\rho^sG,
    \qquad
    F|_{\rho=0}=f.
\end{align*}
Using this expansion, Graham and Zworski \cite{GrahamZworski} define the meromorphic family of \emph{scattering operators} by $S(s)f:=G|_M$. In this formulation, $F|_M$ plays the role of Dirichlet data, while $G|_M$ plays the role of Neumann data. Thus, the scattering operator can be interpreted as a generalized Dirichlet-to-Neumann map.

After a suitable normalization, these give rise to the family of conformally covariant fractional GJMS operators $P_\gamma[g^+,\hat h]$, extending the classical GJMS hierarchy from integer to fractional orders. The normalized scattering operators are of order $2\gamma$, where $s=n/2 + \gamma$ and $\gamma \in (0,n/2)$. These operators have become fundamental objects in conformal geometry and geometric analysis, providing a unified framework for the study of conformally covariant equations of both local and nonlocal type; see \cite{CaseGover} for a survey. They satisfy the conformal covariance law
\begin{align}\label{eq:GJMSconformal}
    P_{\gamma}[g^+,w^{\frac{4}{n-2\gamma}}\hat h](\phi)
    =
    w^{-\frac{n+2\gamma}{n-2\gamma}}
    P_{\gamma}[g^+,\hat h](w\phi),
\end{align}
and have principal symbol equal to that of $(-\Delta_{\hat h})^\gamma$. The fractional GJMS operators interpolate several classical conformally covariant operators. In particular, when $\gamma=1$ one recovers the conformal Laplacian
\begin{align*}
    P_1^{\hat h}
    =
    -\Delta_{\hat h}
    +
    \frac{n-2}{4(n-1)}R_{\hat h},
\end{align*}
while for $\gamma=2$ one obtains the Paneitz operator, whose associated curvature is the classical $Q$-curvature. Similarly, the corresponding fractional curvature is defined by $Q_\gamma^{\hat h} = P_\gamma^{\hat h}(1)$.

The fractional Yamabe problem, first studied by Gonzalez and Qing in the seminal work \cite{GonzQing}, asks whether, given a closed manifold $(M^n,[\hat h])$, one can find a conformal metric to $\hat h$ with constant fractional curvature. Their work is based on the relation between the scattering operators and the Dirichlet-to-Neumann operators of uniformly degenerate elliptic boundary value problems established by Chang and Gonzalez in \cite{ChangGonzalez}. This provides a geometric realization of the extension method of Caffarelli and Silvestre \cite{CaffarelliSilvestre} and allows one to study the problem through a local degenerate elliptic equation on the extension manifold. The fractional Yamabe problem can therefore be viewed as a nonlocal analogue of both the classical Yamabe problem and the Escobar-Yamabe problem, corresponding to exponents $\gamma\in(0,n/2)$.

The conformal covariance property (\ref{eq:GJMSconformal}) allows one to formulate a natural curvature prescription problem. Indeed, given a representative $\hat h\in[\hat h]$, one asks whether there exists another metric in the same conformal class whose associated fractional curvature is constant. Writing the conformal metric as
\[
\hat h_w=w^{\frac{4}{n-2\gamma}}\hat h\in[\hat h],
\]
this is equivalent to finding positive solutions of
\begin{align}\label{eq:fractionalYamabe}
    P_\gamma^{\hat h}(w)
    =
    cw^{\frac{n+2\gamma}{n-2\gamma}} \quad \text{in }M,
\end{align}
for some constant $c$. The equation \eqref{eq:fractionalYamabe} is the corresponding PDE to the fractional Yamabe problem, or simply the $\gamma$-Yamabe problem. As in the classical Yamabe problem and its boundary counterpart, this equation admits a variational formulation. The associated energy is given by the total fractional curvature normalized by the volume
\begin{align*}
    \mathcal I_\gamma(\hat h)
    =
    \frac{\int_M Q_\gamma^{\hat h}\,d\sigma_{\hat h}}
    {\left(\int_M dv_{\hat h}\right)^{\frac{n-2\gamma}{n}}}.
\end{align*}
Restricting this functional to a fixed conformal class by writing $\hat h_w=w^{\frac{4}{n-2\gamma}}\hat h$, one obtains the fractional Yamabe functional
\begin{align*}
    I_\gamma[w,\hat h]
    :=
    \mathcal I_\gamma(\hat h_w)
    =
    \frac{\int_M wP_\gamma^{\hat h}(w)\,d\sigma_{\hat h}}
    {\left(\int_M w^{\frac{2n}{n-2\gamma}}\,d\sigma_{\hat h}\right)^{\frac{n-2\gamma}{n}}}.
\end{align*}
Its infimum over the conformal class defines the fractional Yamabe invariant
\begin{align*}
    \Lambda_\gamma(M,[\hat h])
    =
    \inf\{\mathcal I_\gamma(h):h\in[\hat h]\},
\end{align*}
which is a conformal invariant once the asymptotically hyperbolic metric $g^+$ is fixed.

The operator $P_\gamma^{\hat h}$ is a nonlocal operator on $M$ that is constructed as the Dirichlet-to-Neumann map for a generalized eigenvalue problem on $(X,g^+)$. It was shown in \cite{ChangGonzalez} that this generalized eigenvalue problem is equivalent to a linear degenerate elliptic problem on the compactified manifold $(\overline{X},\bar{g})$. In \cite{ChangGonzalez} they observe that it is possible to find a particular defining function $y$, called the \emph{adapted defining function}, such that when one rewrites the scattering equation obtained in $(\overline{X},\bar{g})$ for the new metric $\bar{g}^* = y^2 g^+$, one obtains a weighted Laplace equation with a nonlinear Robin boundary condition. In this direction, the variational problem is usually studied through this equivalent local extension problem which has the following energy functional
\begin{align}\label{eq:extfunctional}
    I_\gamma^*[U,\bar g^*]
    :=
    \frac{d_\gamma^*\int_X y^a|\nabla U|^2\,dv_{\bar g^*}
    +
    \int_M Q_\gamma^{\hat h}|U|^2\,d\sigma_{\hat h}}
    {\left(\int_M|U|^{2^*}\,d\sigma_{\hat h}\right)^{\frac{n-2\gamma}{n}}},
\end{align}
defined on a suitable weighted Sobolev space, and where $a = 1-2\gamma$ and $2^* = 2n/(n-2\gamma )$.
Critical points of (\ref{eq:extfunctional}) correspond to solutions to the uniformly degenerate elliptic PDE
\begin{align}\label{eq:Extensionpde}
    \begin{cases}
        -\operatorname{div}(y^a \nabla U) = 0 & \text{in } (\overline{X}, \bar{g}^*) \\
        -d_\gamma^* \lim_{y\to 0}y^a \partial_y U + Q_\gamma^{\hat{h}} U = c U^{\frac{n+2\gamma}{n-2\gamma}} & \text{on } M
    \end{cases}.
\end{align}
By the Caffarelli-Silvestre extension, the minimization problems associated to $I_\gamma$ and $I_\gamma^*$ are equivalent.

Using this extension formulation, Gonzalez and Qing proved in \cite{GonzQing} that
\[
\Lambda_\gamma(M,[\hat h])
\leq
\Lambda_\gamma(\mathbb S^n,[h_0]).
\]
where $(\mathbb S^n,[h_0])$ denotes the round sphere. Their proof adapts the classical Aubin-Escobar strategy to the fractional setting by exploiting the extension problem to glue a fractional bubble and construct suitable test functions. They also showed that the $\gamma$-Yamabe invariant is finite and that if the inequality is strict, then there exists a minimizing solution to the $\gamma$-Yamabe problem. Their approach is based on a subcritical approximation argument, in the spirit of \cite{LeeParker,Escobar2}.

\medskip

Since then, the fractional Yamabe problem has been extensively studied due to its rich analytical and geometric structure. Nevertheless, some cases remained open for many years because of the lack of a fractional analogue of the Positive Mass Theorem of Schoen and Yau \cite{SchoenYau}. In \cite{GonzQing}, existence was established in the non-umbilic case. Later, \cite{GonzWang} treated the umbilic, non-locally conformally flat case in high dimensions. In \cite{WeiMusso_Exist}, the authors developed a unified approach covering all cases where the positive mass theorem is not needed. More recently, \cite{MayerNdiaye} settled the locally conformally flat case by adapting the barycenter technique of Bahri and Coron, thus bypassing the need for a positive mass theorem. Progress towards a fractional positive mass theorem has been made in \cite{WeiMusso_Exist,MayerNdiayeAsymp,Almaraz}. For a comprehensive survey on the fractional Yamabe problem, we refer the reader to \cite{Gonz_Survey}.

Regarding the qualitative properties of the solution set of the fractional Yamabe problem, Kim, Musso and Wei proved in \cite{KimMussoWei_comp} that the set of solutions is compact in $C^2(M)$, under a convergence assumption on the scalar curvature of $(X,g^+)$ together with the non-vanishing of the second fundamental form of the conformal infinity $M$. Their proof relies on a delicate blow-up analysis on the boundary manifold $M$. On the other hand, in \cite{KimMussoWei_noncomp} the same authors constructed counterexamples showing that the solution set is, in general, noncompact in sufficiently high dimensions. We also mention the work \cite{GonzPS}, where Fang and Gonz\'alez studied the asymptotic behavior of Palais--Smale sequences associated with fractional Yamabe-type equations. These works provide a rather complete picture of the existence and qualitative behavior of solutions.

Much less is known about the quantitative structure of the set of minimizers and the stability of the fractional Yamabe functional near its minimum. Hence, a natural question one may ask is \emph{if $I_\gamma [w,\hat{h}] $ is close to the infimum $\Lambda_\gamma (M,[\hat{h}])$, does this imply that $w$ is close to a minimizer? Moreover, can this proximity be quantified in terms of the energy deficit?}

For the classical Yamabe problem this issue was addressed initially by Bianchi and Egnell in \cite{BianchiEgnell} giving a quadratic stability estimate for the round sphere. They did it via the classic Sobolev inequality in the Euclidean space, which can be easily translated to the Yamabe setting for the round sphere using stereographic projection. For the general case it was extended to any closed Riemannian manifold not conformally equivalent to the round sphere by Engelstein, Neumayer and Spolaor in \cite{Neumayer}. They combine a Lyapunov-Schmidt reduction and the \L ojasiewicz distance inequality to obtain the following result that completes the answer of quantitative stability for the classical Yamabe problem.
\begin{theoremletter}\label{Thm:Neumayer}
    Let $(M,g)$ be a closed Riemannian manifold of dimension $n\geq 3$ that is not conformally equivalent to the round sphere. Then, there exist constants $C>0$ and $\theta\geq 0$, depending on $(M,g)$, such that
    \begin{align*}
        Q(u)-Y(M,[g]) \geq C d(u,\mathcal{M})^{2+\theta} \quad \forall u \in H^1 (M,\mathbb{R}_+ ),
    \end{align*}
    where $d(-,\mathcal{M})$ is a suitable distance to the set of minimizers $\mathcal{M}$. Moreover, there exists an open dense subset $\mathcal{G}$ in the $C^2$ topology on the space of smooth conformal classes of metrics on $M$ such that if $[g]\in \mathcal{G}$ then $\theta = 0$.
\end{theoremletter}

The part where they show that the metrics that satisfy quadratic stability are dense in the space of conformal classes, what we call generic, relies on a result by R. Schoen \cite{SchoenGeneric}, see also \cite{Anderson}. However, even though \emph{most} smooth conformal classes satisfy quadratic stability, it was proved that the result is sharp in the sense that there are explicit examples of manifolds with positive exponent $\theta$, see for instance \cite{FrankSharpSob,FrankDegStab}.

\medskip

For the Yamabe problem with boundary \cite{Escobar2}, or Escobar-Yamabe problem, the question of quantitative stability was initially addressed by P. T. Ho in \cite{Ho}. He proved that in the model case, the Euclidean ball, the Sobolev trace inequality, studied by Escobar in \cite{Escobarsharp}, also enjoys quadratic stability just as the closed case proven by Bianchi and Egnell. The general version of the quantitative stability was proved by the second author together with Caju and Van Den Bosch in \cite{CajuVDB}, and they use the same techniques as \cite{Neumayer} to show the general result for the Escobar-Yamabe problem on compact manifolds with boundary. More precisely, they deal with the version of Escobar-Yamabe problem that, given a compact Riemannian manifold with boundary $(M,g)$, aims to find a conformal metric with zero scalar curvature and constant mean curvature on the boundary. The PDE that models the problem is
\begin{align}\label{eq:EscobarPDE}
    \begin{cases}
        -\frac{4(n-1)}{n-2} \Delta_g u + R_g u = 0 & \text{in } M \\
        \frac{4(n-1)}{n-2}\frac{\partial u}{\partial \nu} + 2(n-1) H_g u = 2(n-1) c u^{\frac{n}{n-2}} & \text{on } \partial M \, 
    \end{cases},
\end{align}
where $c$ is a constant, $R_g$ and $H_g$ are the scalar and mean curvature with respect to the metric $g$, respectively, and $\nu$ the outward unit normal to $\partial M$. In fact, if $u$ is a solution to (\ref{eq:EscobarPDE}) then the metric $\tilde{g} = u^{4/(n-2)}g$ is scalar flat and has constant mean curvature equal to $c$. As mentioned in \cite{GonzQing}, this is a particular case of the $\gamma$-Yamabe problem when $\gamma = 1/2$. The idea of the proof of this boundary version is to restrict the local analysis around a minimizer of a boundary functional, associated to the Dirichlet-to-Neumann operator, and then extend to the interior using a harmonic extension (in a conformal sense).

As seen in the other Yamabe problems, it is natural to start to study the question by analyzing what happens for the sharp Sobolev inequality on the model case. In that regard, the $\gamma$-Yamabe problem has two Sobolev inequalities that are relevant. The first one, and most natural, is the fractional Sobolev inequality on the Euclidean space $\mathbb{R}^n$. In a work that extends the one from Bianchi and Egnell to the fractional setting, Chen, Frank and Weth proved in \cite{ChenFrankWeth} that there exists a constant $C$, depending on the dimension and $\gamma$, such that
\begin{align}\label{Intro:sphere}
    I_\gamma [w,h_0] - \Lambda_\gamma (\mathbb{S}^n, [h_0] ) \geq C \left( \frac{\inf \{ \|w - v\|_{H^{\gamma } (\mathbb{S}^n)} : v \in \mathcal{M}_{\mathbb{S}^n}\} }{ \|w \|_{H^{\gamma } (\mathbb{S}^n)}}  \right)^2 ,
\end{align}
for any nonnegative $w\in H^{\gamma } (\mathbb{S}^n)$. Here $\mathcal{M}_{\mathbb{S}^n}$ is the family of minimizers on the round sphere. 

The second inequality that is relevant to the problem is the weighted Sobolev trace inequality on the half-space $\mathbb{R}_+^{n+1}$. This inequality is in fact deeply connected to the fractional one on $\mathbb{R}^{n}$ by the Caffarelli-Silvestre extension. Moreover, its minimizers are just convolutions of the fractional bubbles with the Poisson kernel. Using the mentioned extension it is not difficult to prove the quantitative stability for the weighted Sobolev trace inequality, see Section \ref{Section:2}.

In this paper we address the question of quantitative stability of the minimizers of the fractional Yamabe problem on any asymptotically hyperbolic manifold $(X^{n+1},g^+)$ with conformal infinity $(M^n , [\hat{h}])$ of dimension $n\geq 3$, such that $\Lambda_\gamma (M,[\hat{h}]) < \Lambda_\gamma (\mathbb{S}^n , [h_0])$. Just as for the closed case, the minimizers of this general problem are not known explicitly. However, unlike the minimizers of the Sobolev inequality that are explicit but non-compact, here we do have compactness of any minimizing sequence, see Appendix \ref{appendix:B}. The main result of this paper is as follows.

\begin{theorem}[Main Theorem]\label{thm:mainv1}
    Let $\gamma \in (0,1)$ and $(M^n , [\hat{h}])$ be the conformal infinity of an asymptotically hyperbolic manifold $(X^{n+1},g^+ )$ such that $\Lambda_{\gamma} (M,[\hat{h}]) < \Lambda_{\gamma} (\mathbb{S}^n ,[h_{0}])$. Then, there exist constants $C>0$ and $\theta \geq 0$ such that
    \begin{align*}
        I_{\gamma} [w,\hat{h}]  - \Lambda_{\gamma} (M,[\hat{h}]) \geq C d(w,\mathcal{M})^{2+\theta} \quad \forall w \in H^{\gamma}(M,\mathbb{R}_+),
    \end{align*}
    where
    \begin{align*}
        d(w,\mathcal{M}) := \frac{\inf \{\|w-u\|_{H^{\gamma}(M)} : u \in \mathcal{M} \} }{\|w\|_{H^{\gamma}(M)}},
    \end{align*}
    and $\mathcal{M}$ is the set of minimizers of $I_{\gamma}[-,\hat{h}]$.
\end{theorem}
Our approach is, in some sense, the analog of that in \cite{CajuVDB}. In the boundary Yamabe problem, the variational problem is naturally posed on a manifold with boundary, and the Lyapunov--Schmidt reduction is carried out for an effective functional defined on the boundary. In the fractional Yamabe problem, although the original variational problem is formulated on the closed manifold $M$, it is governed by a nonlocal operator. Nevertheless, we show that the same variational strategy can still be carried out in this setting. More precisely, we perform the Lyapunov--Schmidt reduction directly for the fractional Yamabe functional, using the extension characterization of the fractional conformal Laplacian only to obtain the regularity estimates required in the analysis. As a consequence of the equivalence between the two minimization problems, the corresponding quantitative stability estimate for the extension functional follows immediately.

\medskip

The main novelty of this work is the extension of the quantitative stability theory for the Yamabe problem to the fractional setting. Although the overall strategy follows the variational approach developed in \cite{Neumayer,CajuVDB}, several new analytical difficulties arise from the nonlocal nature of the problem. In particular, the Hessian is represented by the fractional conformal Laplacian rather than a uniformly elliptic differential operator, making the classical elliptic theory unavailable. To overcome this difficulty, we combine the extension characterization of Chang and Gonz\'alez with the regularity theory for uniformly degenerate elliptic equations, allowing us to carry out the Lyapunov--Schmidt reduction in the fractional setting. To the best of our knowledge, this provides the first quantitative stability result for the fractional Yamabe problem.

As in the classical and boundary Yamabe problems, the stability estimate does not merely control the distance to the set of minimizers by the energy deficit, but also provides an explicit rate at which this distance decays as the deficit tends to zero. The formulation above, however, depends on the choice of representative metric $\hat{h}\in[\hat{h}]$, since both the functional and the Sobolev norm are written with respect to $\hat{h}$. It is therefore natural to reformulate the estimate in terms of Riemannian metrics, obtaining an equivalent conformally invariant statement. This provides a more geometric interpretation of the quantitative stability inequality and is the formulation stated in Corollary \ref{Cor:conformal}. To do so, we first introduce some additional notation.

\medskip

Indeed, define the following conformally invariant distance between two metrics in the same conformal class
\begin{align*}
    \|\hat{h}_u - \hat{h}_v\| := \left( \int_M |u - v|^{2^*} d\sigma_{\hat{h}} \right)^{1/2^*} ,
\end{align*}
where we identify $\hat{h}_u = u^{4/(n-2\gamma)} \hat{h} \in [\hat{h}]$ and its conformal factor $u \in H^\gamma (M)$. It can be shown that $\|-\|$ is independent of the choice of conformal representative. Also, when $\Lambda_\gamma (M,[\hat{h}]) > 0$, we can define
\begin{align*}
    \|\hat{h}_u - \hat{h}_v\|_* := \left( \int_M (u-v) P_\gamma^{\hat{h}} (u-v) d\sigma_{\hat{h}}   \right)^{1/2} \, ,
\end{align*}
where $\hat{h} \in \mathcal{M}(M,[\hat{h}])$ with unit volume. Again, the quantity $\|-\|_*$ is independent of the choice of conformal representative.

\begin{corollary}[Conformal quantitative stability]\label{Cor:conformal} Let $(M^n , [\hat{h}])$ be the conformal infinity of an asymptotically hyperbolic manifold $(X^{n+1},g^+ )$ such that $\Lambda_{\gamma} (M,[\hat{h}]) < \Lambda_{\gamma} (\mathbb{S}^n ,[h_{0}])$. Then, there exist constants $C>0$ and $\theta \geq 0$, depending on $M$ and $[\hat{h}]$, such that
\begin{align}\label{eq:conformal1}
    \mathcal{I}_\gamma (h) - \Lambda_{\gamma} (M,[\hat{h}]) \geq C \left( \frac{\inf\{\|h - \tilde{h}\|: \tilde{h} \in \mathcal{M}  \} }{\operatorname{vol}_{h}(M)^{1/2^*} }   \right)^{2+\theta} \quad \forall h \in [\hat{h}].
\end{align}
Moreover, when $\Lambda_\gamma (M,[\hat{h}]) > 0$ and the deficit $\mathcal{I}_\gamma (h) - \Lambda_{\gamma} (M,[\hat{h}]) \leq 1$, there exist constants $C>0$ and $\theta \geq 0$, depending on $M$ and $[\hat{h}]$, such that
\begin{align}\label{eq:conformal2}
    \mathcal{I}_\gamma (h) - \Lambda_{\gamma} (M,[\hat{h}]) \geq C \left( \frac{\inf\{\|h - \tilde{h}\|_*: \tilde{h} \in \mathcal{M}  \} }{\operatorname{vol}_{h}(M)^{1/2^*} }   \right)^{2+\theta} \quad \forall h \in [\hat{h}].
\end{align}
\end{corollary}

Just as in Theorem \ref{Thm:Neumayer}, we expect the quadratic stability regime ($\theta=0$) to hold generically. Indeed, quadratic stability follows whenever all minimizers are either non-degenerate or integrable (see Definition \ref{def:integrable}). The latter includes the situations in which the non-uniqueness of minimizers is induced by the symmetries of the problem. This was proved for the Q-curvature case in \cite{Qcurv}.

On the other hand, for particular conformal classes one expects the exponent $\theta$ to be strictly positive. This phenomenon is governed by the \emph{Adams--Simon positivity condition} (see Definition \ref{def:AS}), which is formulated in terms of the first nontrivial homogeneous term in the analytic expansion of the finite-dimensional reduced functional arising from the Lyapunov--Schmidt reduction. As in the classical Yamabe problem, this condition determines the rate at which the energy grows near degenerate minimizers. We establish the following result.

\begin{proposition}[$\text{AS}_p$ implies $\theta>0$]\label{Proposition:AS} Let $(M,[\hat{h}])$ be the conformal infinity of an asymptotically hyperbolic manifold $(X^{n+1},g^+)$, and $p\geq 3$. Let $u_0$ be a non-integrable minimizer of the fractional Yamabe energy and suppose that it satisfies the $\text{AS}_p$ condition. Then, there exists a sequence $(u_k)_{k\in \mathbb{N}} \subset H^\gamma (M)$ with $u_k \to u_0$ in $H^\gamma (M)$ such that 
\begin{align*}
    \lim_{k\to \infty}\frac{I_\gamma [u_k,\hat{h}] -\Lambda_\gamma (M,[\hat{h}])}{\|u_k - u_0\|_{H^\gamma (M)}^{p-\alpha} } = 0 \quad \forall \alpha>0.
\end{align*}
\end{proposition}
\begin{remark}
The existence of explicit degenerate non-integrable minimizers satisfying the $\mathrm{AS}_p$ condition, for some $p\geq 4$, remains an interesting open problem. To the best of our knowledge, no such examples are currently known, even in the boundary Yamabe setting studied in \cite{CajuVDB}. In particular, although the Adams--Simon condition has proved useful in understanding the asymptotic behavior of the Yamabe flow, explicit minimizing metrics satisfying this condition have not yet been constructed.

The situation is even less developed for the fractional Yamabe problem. Existing works on the fractional Yamabe flow do not perform a Lyapunov--Schmidt reduction near degenerate critical points, and consequently notions such as integrability, order of integrability and the Adams--Simon positivity condition have not been investigated in this setting. We expect that understanding these degenerate minimizers will play an important role in the study of the long-time behavior of the fractional Yamabe flow and related variational problems.
\end{remark}

\subsection{Organization of the paper}

In Section \ref{Section:2} we use the result of \cite{ChenFrankWeth} to establish the quantitative stability in the model case. In Section \ref{Section:3} we develop the variational framework needed for the proof of the main results. This includes the Lyapunov--Schmidt reduction and the analysis of the functional in the integrable setting. In Section \ref{Section:4} we first prove the local stability estimate and then deduce the global quantitative stability. Finally, we prove Corollary \ref{Cor:conformal} and Proposition \ref{Proposition:AS}. Appendix \ref{Appendix:A} contains the proof of Lemma \ref{Lemma:LS}, while Appendix \ref{appendix:B} contains the proof of Lemma \ref{lemma:compacntess}.

\subsection{Acknowledgments}

Both authors thank Jie Qing for suggesting this problem and for many helpful discussions. The second author also thanks Rayssa Caju, Hanne Van Den Bosch, and Pedro Gaspar for valuable discussions during his visit to the University of Chile. Finally, the second author thanks Sergio Almaraz for suggesting several useful references related to this work.

\section{The model case}\label{Section:2}

In this section we establish the quantitative stability for the model fractional Yamabe problem. This case deserves a separate treatment for two reasons. First, the set of minimizers is noncompact due to the invariance of the problem under translations and dilations, so the local Lyapunov--Schmidt analysis developed in the following section cannot be applied directly. Second, the model case provides the sharp inequalities that determine the optimal constant in the fractional Yamabe problem on general manifolds, making it the natural starting point for the analysis.

\medskip

As in the classical and boundary Yamabe problems, the strategy is to exploit the explicit characterization of the extremal functions of the relevant Sobolev inequalities. In the closed Yamabe problem these extremals arise from the sharp Sobolev inequality on $\mathbb{R}^n$, while in the boundary Yamabe problem they are related to the sharp Sobolev trace inequality on the half-space. The fractional Yamabe problem follows the same philosophy: the extremal functions are first identified in the Euclidean setting and then transferred to the geometric problem through the conformal covariance of the fractional conformal Laplacian.

\medskip

The approach of González and Qing \cite{GonzQing} closely parallels Escobar's treatment of the boundary Yamabe problem. A fundamental ingredient is the sharp weighted Sobolev trace inequality associated with the Caffarelli--Silvestre extension, whose extremal functions coincide with the extensions of the extremals of the sharp fractional Sobolev inequality on $\mathbb{R}^n$. We therefore begin by recalling these two inequalities and proving their corresponding quantitative stability estimates. Recall that $2^* = 2n/(n-2\gamma)$, for any $\gamma \in (0,n/2)$.

\subsection{Fractional Sobolev inequality}

The first ingredient is the sharp fractional Sobolev inequality on $\mathbb{R}^n$. In \cite{ChenFrankWeth}, Chen, Frank and Weth proved a quantitative stability estimate of this inequality and, by means of stereographic projection, transferred it to an equivalent stability result on the round sphere. Since the model case of conformal infinity of the fractional Yamabe problem is precisely the round sphere, these results provide the starting point for our analysis. We first recall the sharp fractional Sobolev inequality.

\begin{theorem}
    Let $\gamma\in (0,n/2)$. Then for all $w \in \dot{H}^{\gamma} (\mathbb{R}^n )$ we have
    \begin{equation*}
        \|w \|_{L^{2^{*}}(\mathbb{R}^n )}^2 \leq S(n,\gamma) \|(-\Delta )^{\gamma/2}w \|_{L^{2} (\mathbb{R}^n )}^2 = S(n,\gamma) \int_{\mathbb{R}^n } w (-\Delta )^{\gamma}w  dx,
    \end{equation*}
    where
    \begin{align*}
        S(n,\gamma) = \frac{\Gamma ((n-2\gamma)/2)}{\Gamma ((n+2\gamma)/2)} |\operatorname{vol}(\mathbb{S}^n )|^{-\frac{2\gamma}{n}}.
    \end{align*}
    We have equality if and only if
    \begin{align}\label{eq:fractionalbubbles}
        w(x) = c\left( \frac{\mu}{|x-x_0 |^2 + \mu^2 }  \right)^{\frac{n-2\gamma}{2}} \quad x\in \mathbb{R}^n,
    \end{align}
    for $c\in \mathbb{R}$, $\mu > 0$ and $x_0 \in \mathbb{R}^n$. We denote the set of these functions as $\mathcal{M}_{\mathbb{R}^n} $.
\end{theorem}

The explicit characterization of the extremal functions allows one to ask whether almost extremals must be quantitatively close to the family of bubbles. This was answered by Chen, Frank and Weth in \cite{ChenFrankWeth}, who proved the following quantitative stability estimate.

\begin{theorem}[Chen, Frank and Weth \cite{ChenFrankWeth}]\label{Thm:FractionalStab}
    There exists a positive constant $C = C (n,\gamma) > 0$ such that
    \begin{equation*}
         \int_{\mathbb{R}^n } w (-\Delta )^{\gamma}w  dx - \frac{1}{S(n,\gamma) }\left( \int_{\mathbb{R}^n } |w|^{2^{*}} dx  \right)^{\frac{2}{2^{*}}} \geq C \text{dist}(w,\mathcal{M}_{\mathbb{R}^n} )^2 \quad \forall w\in \dot{H}^{\gamma } (\mathbb{R}^n),
    \end{equation*}
    where
    \begin{align*}
        \text{dist}(w,\mathcal{M}_{\mathbb{R}^n} ) := \inf \{\|w-u \|_{\dot{H}^{\gamma}(\mathbb{R}^n)} : u \in \mathcal{M}_{\mathbb{R}^n}  \},
    \end{align*}
    and $\mathcal{M}_{\mathbb{R}^n}$ is the set of extremal functions for the fractional Sobolev inequality.
\end{theorem}

The Euclidean stability estimate can be transferred to the round sphere through stereographic projection. This is carried out in \cite{ChenFrankWeth}, yielding an equivalent stability inequality for functions on $\mathbb{S}^n$. The same strategy is used in the classical setting to pass from the Bianchi--Egnell inequality to the sphere; see, for instance, \cite{Schoenlectures,Neumayer}. Since this argument is standard, we simply state the resulting theorem.

\begin{theorem}[Chen, Frank and Weth \cite{ChenFrankWeth}]\label{thm:ChenFrankWethSphere}
    There exists a constant $C=C (n,\gamma) >0$ such that
    \begin{equation*}
        \frac{\| w \|_{H^\gamma (\mathbb{S}^n)}^2 }{\|w\|_{L^{2^*} (\mathbb{S}^n)}^2} - \frac{1}{S(n,\gamma)} \geq C \frac{\text{dist}(w,\mathcal{M}_{\mathbb{S}^n} )^2}{\|w\|_{L^{2^*} (\mathbb{S}^n)}^2}
         \quad \forall w\in H^{\gamma } (\mathbb{S}^n),
    \end{equation*}
    where $\mathcal{M}_{\mathbb{S}^n } = \{v \in H^\gamma (\mathbb{S}^n) : \mathcal{P}v \in \mathcal{M}_{\mathbb{R}^n} \} $ and $\mathcal{P}$ is the isometric isomorphism given by the stereographic projection.
\end{theorem}

The previous theorem is formulated in terms of the Sobolev norm introduced through stereographic projection. To use it in the fractional Yamabe setting, one only needs to identify this norm with the energy appearing in the Yamabe functional on the round sphere. Indeed, we use the fact that the normalized scattering operator on the sphere $P_{\gamma}^{h_0 }$ is the pull back of $(-\Delta_{\mathbb{R}^n})^\gamma$ via the stereographic projection, see \cite{ChenFrankWeth,NirembergProb}, we have that
    \begin{align}\label{eq:normsphere}
        \|w\|_{H^{\gamma} (\mathbb{S}^n )}^2 = \int_{\mathbb{R}^n } \mathcal{P}w (-\Delta_{\mathbb{R}^n})^\gamma (\mathcal{P}w) dx = \int_{\mathbb{S}^n } w P_{\gamma}^{h_0 } (w) d\sigma_{h_0} .
    \end{align}

This immediately yields the following quantitative stability estimate for the model case.
\begin{proposition}
    Let $(\mathbb{S}^n , h_0 )$ be the round sphere with $n\geq 3$. Then, there exists a constant $C>0$ such that for all $w \in H^{\gamma}(\mathbb{S}^n)$
    \begin{align*}
        I_\gamma [w,h_0] - \Lambda_\gamma (\mathbb{S}^n, [h_0] ) \geq C d(w,\mathcal{M}_{\mathbb{S}^n })^2,
    \end{align*}
    where $\mathcal{M}_{\mathbb{S}^n}$ is the set of minimizers of $I_\gamma [-, h_0]$, and
    \begin{align*}
        d(w,\mathcal{M}_{\mathbb{S}^n}) := \frac{\inf \{\|w - v \|_{H^{\gamma} (\mathbb{S}^n) } : \mathcal{P}v \in \mathcal{M}_{\mathbb{R}^n} \}}{\|w\|_{ H^{\gamma} (\mathbb{S}^n)  }}.
    \end{align*}
\end{proposition}

\begin{proof}
    Let $w \in H^{\gamma} (\mathbb{S}^n )$, then by Theorem \ref{thm:ChenFrankWethSphere} we have the following inequality
        \begin{equation*}
            \frac{\| w \|_{H^{\gamma} (\mathbb{S}^n )}^2 }{\|w\|_{L^{2^*} (\mathbb{S}^n) }^2} - \frac{1}{S(n,\gamma)} \geq C \frac{\text{dist}(w,\mathcal{M}_{\mathbb{S}^n} )^2}{\|w\|_{L^{2^*} (\mathbb{S}^n) }^2}.
        \end{equation*}
    Noting that $\Lambda_\gamma (\mathbb{S}^n, [h_0] ) = S(n,\gamma)^{-1}$ and the expression (\ref{eq:normsphere}), we have
    \begin{equation*}
        I_\gamma [w,h_0] - \Lambda_\gamma (\mathbb{S}^n, [h_0] ) \geq C \frac{\text{dist}(w,\mathcal{M}_{\mathbb{S}^n} )^2}{\|w\|_{L^{2^*} (\mathbb{S}^n) }^2}\, .
    \end{equation*}
    To conclude we use the fractional Sobolev embedding to obtain the desired normalized distance
    \begin{align*}
        I_{\gamma}[w,h_0 ] - \Lambda_\gamma (\mathbb{S}^n, [h_0] ) &\geq C \frac{\text{dist}(w,\mathcal{M}_{\mathbb{S}^n} )^2}{\|w\|_{L^{2^*} (\mathbb{S}^n) }^2} \geq C d(w,\mathcal{M}_{\mathbb{S}^n })^2.
    \end{align*}
\end{proof}

\subsection{Weighted Sobolev trace inequality}
We now consider the weighted Sobolev trace inequality, which relates the intrinsic fractional Sobolev inequality on $\mathbb{R}^n$ to its extension formulation on $\mathbb{R}_+^{n+1}$. We first recall the sharp inequality and its extremal functions.

\begin{theorem}[Gonzalez and Qing \cite{GonzQing}]
    Let $\gamma \in (0,1)$ and $a= 1-2\gamma$. There exists a constant $\bar{S} (n,\gamma)$ such that, for every $U\in \dot{W}^{1,2}(\mathbb{R}_{+}^{n+1}, y^a )$ with trace $TU =w$, the following inequality holds
    \begin{align*}
        \|w \|_{L^{2^{*}}(\mathbb{R}^n )}^2 \leq \bar{S}(n,\gamma) \int_{\mathbb{R}_{+}^{n+1}} y^a |\nabla U|^2 dxdy \, ,
    \end{align*}
    where $\bar{S}(n,\gamma) = d_\gamma^* S(n,\gamma)$. Equality holds if and only if
    \begin{align*}
        w(x) = c \left(\frac{\mu}{\mu^2 + |x-x_0 |^2}\right)^{\frac{n-2\gamma}{2}} \quad x \in \mathbb{R}^n,
    \end{align*}
    for $c \in \mathbb{R}$, $\mu >0$ and $x_0 \in \mathbb{R}^n$, and $U$ is the Poisson extension of $w$ given by the convolution with the Poisson kernel.
\end{theorem}

The extremals of the weighted Sobolev trace inequality are precisely the weighted Poisson extensions of the fractional bubbles (\ref{eq:fractionalbubbles}). Explicitly,
    \begin{align*}
        \mathcal{M}_{\mathbb{R}_{+}^{n+1}} = \{K * v : v \in \mathcal{M}_{\mathbb{R}^n} \}.
    \end{align*}
Therefore, Theorem \ref{Thm:FractionalStab}, together with the extension characterization, yield the corresponding stability estimate for the weighted trace inequality.

\begin{proposition}[Stability of the weighted Sobolev trace inequality]\label{Propositio:Stabilitytrace} Let $\gamma \in (0,1)$ and $a= 1-2\gamma$. Then, there exists a constant $C>0$ such that for any $U\in \dot{W}^{1,2}(\mathbb{R}_{+}^{n+1}, y^a )$ with trace $TU =w$
\begin{align*}
    \int_{\mathbb{R}_{+}^{n+1}} y^a |\nabla U|^2 dxdy - \frac{1}{\bar{S}(n,\gamma)} \|w \|_{L^{2^{*}}(\mathbb{R}^n )}^2  \geq C \text{dist} (U,\mathcal{M}_{\mathbb{R}_+^{n+1}})^2 \, ,
\end{align*}
where
\begin{align*}
    \text{dist} (U,\mathcal{M}_{\mathbb{R}_+^{n+1}}) = \inf \{\|U - V\|_{\dot{W}^{1,2}(\mathbb{R}^{n+1}_+,y^a)} : V \in \mathcal{M}_{\mathbb{R}^{n+1}_+ } \} .
\end{align*}
\end{proposition}
\begin{proof}
    Let $U\in \dot{W}^{1,2}(\mathbb{R}_{+}^{n+1}, y^a )$ with trace $w \in \dot{H}^\gamma (\mathbb{R}^n)$. We can always decompose
    \begin{align*}
        U = U_w + U_0 \, ,
    \end{align*}
    where $U_w = K*w$ and $TU_0 = 0$. Note that $U_w$ is the unique energy-minimizing extension among all functions with trace $w$. Since $\text{div}(y^a \nabla U_w ) = 0$ and $TU_0 = 0$, both terms are orthogonal in the weighted Dirichlet form. Consequently,
    \begin{align*}
        \int_{\mathbb{R}_{+}^{n+1}} y^a |\nabla U|^2 dxdy =& \int_{\mathbb{R}_{+}^{n+1}} y^a |\nabla U_w|^2 dxdy + \int_{\mathbb{R}_{+}^{n+1}} y^a |\nabla U_0|^2dxdy + 2 \int_{\mathbb{R}_{+}^{n+1}} y^a \langle \nabla U_0, \nabla U_w  \rangle  dxdy \\
        =& \int_{\mathbb{R}_{+}^{n+1}} y^a |\nabla U_w|^2 dxdy + \int_{\mathbb{R}_{+}^{n+1}} y^a |\nabla U_0|^2dxdy  \\
        =& \frac{1}{d_\gamma^*} \int_{\mathbb{R}^n} w (-\Delta)^\gamma w \, dx+ \int_{\mathbb{R}_{+}^{n+1}} y^a |\nabla U_0|^2 dxdy .
    \end{align*}
    Thus, we have
    \begin{align*}
        \int_{\mathbb{R}_{+}^{n+1}} y^a |\nabla U|^2  dxdy  - \frac{\|w \|_{L^{2^{*}}(\mathbb{R}^n )}^2}{\bar{S}(n,\gamma)}  =& \frac{1}{d_\gamma^*} \int_{\mathbb{R}^n} w (-\Delta)^\gamma w \, dx + \int_{\mathbb{R}_{+}^{n+1}} y^a |\nabla U_0|^2 dxdy - \frac{\|w \|_{L^{2^{*}}(\mathbb{R}^n )}^2 }{\bar{S}(n,\gamma)}  \\
        =& \frac{1}{d_\gamma^*} \int_{\mathbb{R}^n} w (-\Delta)^\gamma w \, dx + \int_{\mathbb{R}_{+}^{n+1}} y^a |\nabla U_0|^2 dxdy -  \frac{\|w \|_{L^{2^{*}}(\mathbb{R}^n )}^2 }{d_\gamma^* S(n,\gamma) } \\
        \geq& \frac{C}{d_\gamma^*}\,  \text{dist}(w,\mathcal{M}_{\mathbb{R}^n} )^2 + \|U_0\|_{\dot{W}^{1,2}(\mathbb{R}_{+}^{n+1},y^a)}^2 \\
        =& \frac{C}{d_\gamma^*} \, \inf \{ \|w - v\|_{\dot{H}^\gamma (\mathbb{R}^n)} : v \in \mathcal{M}_{\mathbb{R}^n}   \}^2 + \|U_0\|_{\dot{W}^{1,2}(\mathbb{R}_{+}^{n+1},y^a)}^2 \\
        =& C \, \inf \{ \|U_w - V\|_{\dot{W}^{1,2}(\mathbb{R}_{+}^{n+1},y^a)} :  V \in \mathcal{M}_{\mathbb{R}_+^{n+1}}  \}^2 + \|U_0\|_{\dot{W}^{1,2}(\mathbb{R}_{+}^{n+1},y^a)}^2 \\
        \geq& \tilde{C} \, \inf \{ \|U_w + U_0 - V\|_{\dot{W}^{1,2}(\mathbb{R}_{+}^{n+1},y^a)} :  V \in \mathcal{M}_{\mathbb{R}_+^{n+1}}  \}^2 \\
        =& \tilde{C} \, \inf \{ \|U - V\|_{\dot{W}^{1,2}(\mathbb{R}_{+}^{n+1},y^a)} :  V \in \mathcal{M}_{\mathbb{R}_+^{n+1}}  \}^2 \\
        =& \tilde{C} \, \text{dist}(U,\mathcal{M}_{\mathbb{R}_+^{n+1}} )^2 \, ,
    \end{align*}
    where $\tilde{C} = \min \{C,1\}$, and we can put $U_0$ inside the homogeneous norm because the cross term
    \begin{align*}
        \int_{\mathbb{R}_{+}^{n+1}} y^a \langle \nabla (U_w - V),\nabla U_0\rangle dxdy = 0 \, ,
    \end{align*}
    vanishes since $V \in \mathcal{M}_{\mathbb{R}_+^{n+1}} $. We have proved the desired result.
\end{proof}

It is worth noting that, if we consider the upper half-space with the hyperbolic metric, the $y$ coordinate corresponds to the adapted defining function. This means that the previous result implies directly, just by adapting the language, the quadratic stability of the model case for the extension formulation of the fractional Yamabe problem.

\begin{remark}
The converse implication also holds. Indeed, restricting the weighted Sobolev trace inequality to harmonic extensions and arguing exactly as in the proof of Proposition \ref{Propositio:Stabilitytrace}, one recovers the quantitative stability of the fractional Sobolev inequality. Therefore, the two stability inequalities are equivalent through the Caffarelli-Silvestre extension.
\end{remark}

\section{Properties of the fractional Yamabe energy and Lyapunov-Schmidt reduction}\label{Section:3}
In this section, we introduce the variational framework for the fractional Yamabe functional and carry out a Lyapunov-Schmidt reduction near a fixed minimizer. We first describe the functional setting and compute the corresponding first and second variations under the normalization constraint. We then state the whole statement of the reduction, leaving the proof to the Appendix \ref{Appendix:A}.

\medskip

Throughout this section, fix $\gamma \in (0,1)$ and let $(X^{n+1},g_+)$ be a conformally compact asymptotically hyperbolic manifold with conformal boundary $(M^n,[\hat{h}])$. We fix a representative $\hat{h}\in[\hat{h}]$ and let $y $ denote the adapted defining function, so that the extension problem with respect to the metric $\bar{g}^* := y^2 g^+$ is the one given in (\ref{eq:Extensionpde}). So, from now on we work on the corresponding compactified manifold $(\overline X , \bar{g}^*)$, and set $a := 1-2\gamma$. Unless otherwise stated, all geometric quantities on $M$ are computed with respect to $\hat{h}$, while those on $\overline{X}$ are computed with respect to $\bar{g}^*$.

\medskip

To simplify notation from now on we will denote $u := TU$, for any $U \in W^{1,2}(X,y^a)$, and similarly for any function in $W^{1,2}(X,y^a)$. Also, since in the following analysis we do not change the metric we will just write $I_\gamma [-,\hat{h}] $ when referring to $ I_\gamma [-]$, and similarly with $I_\gamma^*[-]$ when referring to $I_\gamma^* [-,\bar{g}^*]$.

\subsection{Variational properties of $I_\gamma$}
To gain compactness of the problem we we work on the
normalized constraint set
\begin{align*}
    \mathcal{B} := \{u \in H^\gamma (M,\mathbb{R}_+) :  \, \|u\|_{L^{2^* }} =1\}.
\end{align*}
Near a positive minimizer, we work in a sufficiently small
$C^{2\gamma+\alpha}$-neighborhood, where positivity is preserved. Elements of this set represent the conformal factors that give metrics conformal to $\hat{h}$ with unit volume in $M$. Since the admissible variations are those preserving the normalization constraint to first order, we perform the variational analysis on the tangent space of the constraint manifold. Thus, for any $u \in \mathcal{B}$, we define
\begin{align}\label{eq:tangentsp}
    T_u \mathcal{B} = \left\{ \phi \in H^\gamma (M) : \int_M u^{\frac{n+2\gamma}{n-2\gamma}} \phi \, d\sigma_{\hat{h}} = 0 \right\},
\end{align}
and its associated $L^2$-orthogonal projection onto $T_u \mathcal{B}$
\begin{align*}
    \pi_{T_u \mathcal{B}} : H^\gamma (M) &\to T_u \mathcal{B} \subset H^\gamma (M) \\
    \phi &\mapsto \phi - \langle u^{\frac{n+2\gamma}{n-2\gamma}} , \phi \rangle u\, .
\end{align*}

We start by giving the expressions for the first and second variation of $I_\gamma$ on $\mathcal{B}$. As is usual in this context, we will denote $\nabla_{\mathcal{B}} I_\gamma [u] (-) := \nabla I_\gamma [u] ( \pi_{T_u \mathcal{B}} (-)) $, and similar for $\nabla_{\mathcal{B}}^2 I_\gamma [u] (-,-):= \nabla^2 I_\gamma [u] (\pi_{T_u \mathcal{B}}(-),\pi_{T_u \mathcal{B}}(-))$.

\begin{lemma}\label{lemma:Variationalppties}
    For every $u \in \mathcal{B}$, the first variation of $I_{\gamma}$ on $\mathcal{B}$ is given by
    \begin{align*}
        \frac{1}{2} \nabla_{\mathcal{B}} I_{\gamma} [u] (\phi) = \int_M \phi P_\gamma^{\hat{h}}(u) d\sigma_{\hat{h}},
    \end{align*}
    for every $\phi \in H^\gamma (M) $. The second variation over $\mathcal{B}$ is given by
    \begin{align*}
        \frac{1}{2} \nabla_{\mathcal{B}}^2 I_{\gamma} [u](\phi, \psi) = \int_M \phi P_\gamma^{\hat{h}}(\psi) d\sigma_{\hat{h}} - (2^* - 1)I_\gamma [u] \int_M u^{\frac{4\gamma}{n-2\gamma}}\phi \psi d\sigma_{\hat{h}}.
    \end{align*}
    for every $\phi,\psi \in H^{\gamma}(M) $. Moreover, the following properties hold
    \begin{enumerate}
        \item\label{lemma:variat_1} For $\alpha \in (0,1)$ such that $2\gamma + \alpha \notin \mathbb{N}$, the map $u \mapsto \frac{\nabla_{\mathcal{B}}^2 I_\gamma [u](\phi,-)}{\|\phi\|_{C^{2\gamma + \alpha}} }$ is locally continuous from $C^{2\gamma + \alpha}(M) \cap \mathcal{B} \to C^{ \alpha}(M)$ with a modulus of continuity uniform over $\phi\in C^{2\gamma + \alpha}(M)$.
        \item\label{lemma:variat_2} The function $u \mapsto \frac{\nabla_{\mathcal{B}}^2 I_\gamma [u](\phi,\psi)}{\|\phi\|_{H^\gamma } \|\psi\|_{H^\gamma  }}$ is a locally continuous function from $ \mathcal{B} \to \mathbb{R}$ with a modulus of continuity uniform over $\phi,\psi \in H^\gamma (M)$.
    \end{enumerate}
\end{lemma}

\begin{proof}
    Let $u \in \mathcal{B}$ and $\phi \in H^\gamma (M) $. Then
    \begin{align*}
        \nabla I_{\gamma} [u](\phi) = &2 \left(\int_M \phi P_\gamma^{\hat{h}}(u) d\sigma_{\hat{h}} \right) \left(\int_M u^{2^* }d\sigma_{\hat{h}} \right)^{-\frac{2}{2^* }} \\
        &- 2 I_\gamma [u]  \left(\int_M u^{2^* } d\sigma_{\hat{h}} \right)^{-1} \int_M u^{\frac{n+2\gamma}{n-2\gamma}}\phi d\sigma_{\hat{h}} \,.
    \end{align*}
    Using the fact that $u \in \mathcal{B}$ and imposing that $\phi \in T_u \mathcal{B}$, recall the definition (\ref{eq:tangentsp}), we obtain
    \begin{align*}
        \frac{1}{2} \nabla_{\mathcal{B}} I_{\gamma} [u](\phi) = \int_M \phi P_\gamma^{\hat{h}}(u) d\sigma_{\hat{h}} .
    \end{align*}
    For the second variation, for $\phi,\psi \in H^\gamma (M)$ and, again using (\ref{eq:tangentsp}), we obtain
    \begin{align*}
        \frac{1}{2}\nabla_{\mathcal{B}}^2 I_\gamma [u] [\phi,\psi] = \int_M \phi P_\gamma^{\hat{h}} (\psi) d\sigma_{\hat{h}} - (2^* -1 )I_\gamma [u] \int_M u^{\frac{4\gamma}{n-2\gamma}} \phi \psi d\sigma_{\hat{h}}.
    \end{align*}
    
    Let $u\in \mathcal{B}$, then operator associated to the bilinear form $\nabla_{\mathcal{B}}^2 I_\gamma [u] $ is given by
    \begin{align*}
        \mathcal{L}_u \phi := P_\gamma^{\hat{h}} (\phi) - (2^* -1) I_\gamma [u] u^{2^* -2} \phi
    \end{align*}
    for $\phi \in H^\gamma (M)$. Next, to prove the continuity of the second variation we proceed just as \cite{Neumayer}. Indeed, after some basic algebra we get for any $\phi\in H^\gamma (M)$
    \begin{align*}
        \| \mathcal{L}_u \phi - \mathcal{L}_v \phi \|_X \leq C \left(|I_\gamma [u] - I_\gamma [v] | \, \|u^{2^* - 2}\phi\|_X  + |I_\gamma [v] | \, \|\phi |u^{2^* - 2} - v^{2^* - 2} | \|_X \right) \, ,
    \end{align*}
    where $X = C^{ \alpha} (M) \text{ or } X =H^{-\gamma} (M)$. If $X = C^{ \alpha} (M)$ using that $s\mapsto s^{2^* - 2}$ is continuous for the second term and the continuity of $I_\gamma [-]$ in $C^{2\gamma + \alpha} (M)$ for the first term, we get
    \begin{align*}
        \|\mathcal{L}_u \phi - \mathcal{L}_{v} \phi \|_{C^{\alpha}(M) } \leq \omega (\|u - v\|_{C^{2\gamma + \alpha }(M) } ) \|\phi\|_{C^{ 2\gamma + \alpha}(M)} \, ,
    \end{align*}
    for $\omega$ some modulus of continuity. If $X =H^{-\gamma}(M)$, thanks to Hölder and the Sobolev embedding $H^\gamma (M) \hookrightarrow L^{2^*} (M)$, we can estimate
    \begin{align*}
        \|u^{2^* - 2}\phi\|_{H^{-\gamma} (M) } \leq C \|u^{2^* -2 } \|_{L^{n/(2\gamma) } (M)} \|\phi\|_{L^{2^*} (M) } = C \|u\|_{L^{2^*} (M)}^{\frac{4\gamma}{n-2\gamma}} \|\phi\|_{L^{2^*} (M) }\leq C  \|\phi\|_{H^\gamma (M)} \, .
    \end{align*}
    This estimate together with the continuity of $I_\gamma [-]$ in $H^\gamma (M)$ yields 
    \begin{align*}
        |I_\gamma [u] - I_\gamma [v]| \, \|u^{2^* - 2} \phi \|_{H^{-\gamma} (M)} \leq \omega (\|u - v\|_{H^\gamma (M)  })  \|\phi\|_{H^\gamma (M)} \, .
    \end{align*}
    For the second term we apply similar arguments to get
    \begin{align*}
        \|\phi |u^{2^* - 2} - v^{2^* - 2}| \|_{H^{-\gamma} (M) } \leq \omega (\|u -v\|_{H^\gamma (M) } ) \|\phi\|_{H^\gamma (M)}\, ,
    \end{align*}
    for some modulus of continuity $\omega$. To complete the proof we need to show that the map $u \mapsto \pi_{T_u \mathcal{B}}$ is continuous from $C^{2\gamma + \alpha} \cap \mathcal{B} \to L(C^{2\gamma + \alpha}, C^{2\gamma + \alpha})$. This is completely analogous to the work done in \cite{Neumayer}.
\end{proof}
We now introduce some notation. Given a function $u \in \mathcal{B}$, denote the normalized ball in $H^\gamma (M)$ centered at $u$ of radius $\delta>0$ by
\begin{align*}
    \mathcal{B}(u,\delta) := \{w\in \mathcal{B} :  \|w-u\|_{H^\gamma (M)} \leq \delta  \}.
\end{align*}

\subsection{Local analysis near a minimizer}

We now establish the local stability estimate around a fixed minimizer by means of a Lyapunov--Schmidt reduction. Recall that
\begin{align*}
    \mathcal{M}_1
    :=
    \{u\in\mathcal B:
    I_\gamma[u,\hat h]
    =
    \Lambda_\gamma(M,[\hat h])\},
\end{align*}
and fix $u_*\in\mathcal M_1$. Throughout this subsection we study the behavior of the fractional Yamabe functional $I_\gamma$ in a neighborhood of $u_*$. The goal is to separate the nondegenerate and degenerate directions of the second variation, reducing the infinite-dimensional variational problem to a finite-dimensional one. This construction is analogous to the Lyapunov--Schmidt reductions used in \cite{Neumayer,CajuVDB} for the classical and boundary Yamabe problems. Since the arguments are essentially the same, we only discuss the points where the fractional setting requires additional analysis.

\medskip

The main difference with respect to the classical Yamabe problem is that the linearized operator is no longer a uniformly elliptic differential operator acting on the closed manifold $M$. Instead, it is given by the fractional conformal Laplacian $P_\gamma^{\hat h}$, a nonlocal pseudodifferential operator of order $2\gamma$. To recover the local elliptic framework needed for the Lyapunov--Schmidt reduction, we use the extension characterization of $P_\gamma^{\hat h}$ established in \cite{ChangGonzalez}. This realizes the nonlocal equation as a uniformly degenerate elliptic boundary value problem on the extension manifold $X$. Since $a\in(-1,1)$, the weight $y^a$ belongs to the Muckenhoupt class $\mathcal A_2$, and the corresponding weighted elliptic theory provides the Fredholm and Schauder estimates required below.

\medskip

The local behavior of the energy near $u_*$ is encoded by the Hessian $\nabla_{\mathcal B}^2I_\gamma[u_*]$, whose associated linearized operator is
\begin{align*}
    \mathcal L_{u_*}
    :=
    P_\gamma^{\hat h}
    -
    (2^*-1)\Lambda_\gamma(M,[\hat h])
    u_*^{\frac{4\gamma}{n-2\gamma}}.
\end{align*}
Since $u_*$ is a minimizer, the associated quadratic form is nonnegative. Moreover, $\mathcal L_{u_*}$ is a self-adjoint Fredholm operator and therefore has finite-dimensional kernel $K:=\ker\bigl(\nabla_{\mathcal B}^2I_\gamma[u_*]\bigr)$. We write $l:=\dim K$, denote by $\pi_K$ the orthogonal projection onto $K$, and let $K^\perp$ be its orthogonal complement in $T_{u_*}\mathcal B$ with respect to the $L^2(M)$ inner product. We are now ready to perform the Lyapunov--Schmidt reduction. The following result is the analogue of \cite[Lemma~2.2]{Neumayer} in the fractional setting. The proof follows the same strategy, the main additional ingredient being the PDE theory discussed above.

\begin{lemma}\label{Lemma:LS}
Let $(M^n,[\hat h])$ be the conformal infinity of an asymptotically
hyperbolic manifold $(X^{n+1},g^+)$, and fix
$u_*\in\mathcal M_1$. Then there exist an open neighborhood $\mathcal U\subset K$ of $0$
and a smooth map
\begin{align*}
    F:\mathcal U\longrightarrow K^\perp
\end{align*}
such that $F(0)=0$ and $\nabla F(0)=0$, and the following properties hold.
\begin{enumerate}
    \item\label{LS1}
    Let $q:\mathcal{U} \to \mathbb{R}$ be the function defined by $q(\phi) := I_\gamma [u_* + \phi + F(\phi)]$. Then, we have
    \begin{align*}
        \mathcal S
        :=
        \left\{ u_* + \phi + F(\phi): \phi \in \mathcal{U}
        \right\} \subset\mathcal{B},
    \end{align*}
    and
    \begin{align*}
    \nabla_{\mathcal B} I_\gamma
        [u_* + \phi + F(\phi)]
        &=
        \pi_K \nabla_{\mathcal B} I_\gamma
        [u_* + \phi + F(\phi)]\\
        &=
        \nabla q(\phi).
    \end{align*}
    Moreover, $q$ is real analytic.
    \item\label{LS2}
    There exists $\delta>0$, depending on $u_*$, such that for every
    $u\in  \mathcal{B} (u_*,\delta)$, we have $\pi_K (u-u_*) \in \mathcal{U}$. Moreover, if $u\in \mathcal{B} (u_*,\delta)$ is a critical point of
    $I_\gamma$, then
    \begin{align*}
        u = u_* + \pi_K (u-u_*) + F(\pi_K (u-u_*)).
    \end{align*}
    \item\label{LS3}
    Fix $\alpha\in(0,1)$ such that $2\gamma+\alpha\notin\mathbb N$, then there exists $C>0$ such that, for every
    $\phi\in\mathcal U$ and every $\psi\in K$,
    \begin{align*}
        \|\nabla F(\phi)[\psi]\|_{C^{2\gamma+\alpha}(M)}
        \leq
        C\|\psi\|_{C^\alpha(M)}.
    \end{align*}
\end{enumerate}
\end{lemma}

\subsection{Integrability condition}

Associated to the Lyapunov--Schmidt reduction is the notion of integrability, which characterizes whether every infinitesimal deformation in the kernel of the linearized operator is generated by a genuine family of critical points. As in the classical Yamabe problem, this property determines the local structure of the set of critical points.

\begin{definition}[Integrability]\label{def:integrable}
A critical point $u_* \in \mathcal{B}$ is said to be \emph{integrable} if, for every $\phi \in K$, there exists a one-parameter family $(u_t)_{t\in(-\delta,\delta)} \subset \mathcal{B}$ of critical points of $I_\gamma$ such that
\begin{align*}
    u_0=u_*, \qquad \left.\frac{d}{dt}\right|_{t=0}u_t=\phi.
\end{align*}
\end{definition}

\begin{lemma}[$I_\gamma$ in the integrable setting]\label{lemma:integrable}
Let $u_* \in \mathcal{M}_1$. Then $u_*$ is integrable if and only if the reduced functional $q$ is constant in a neighborhood of $0 \in K$. In particular, if $u_* \in \mathcal{M}_1$ is an integrable minimizer, then
\begin{align*}
    \mathcal{M}_1 \cap \mathcal{B} (u_*,\delta)=\mathcal{S}.
\end{align*}
\end{lemma}

The proof is identical to that of the classical Yamabe problem. Indeed, once the Lyapunov--Schmidt reduction has been established, the reduced functional $q$ is a real analytic function on the finite-dimensional space $K$. The argument therefore becomes purely finite-dimensional and carries over verbatim; see \cite[Lemma~2.4]{Neumayer} or \cite[Lemma 3.4]{CajuVDB}.

\section{Quantitative stability on general manifolds}\label{Section:4}
To prove the main result we prove a local version first, and then conclude by a compactness argument the global result.

\subsection{Local stability}
In this section we establish a local version of Theorem \ref{thm:mainv1}, which is Proposition \ref{Proposition:localstab}. For that we need a quantity that measures how far is a function from minimizers that are $\delta$-close to a fixed one $u_*$. Given $\delta>0$ and a fixed $u_* \in \mathcal{M}_1$ we define
\begin{align*}
    d_\delta (u,\mathcal{M}_1 ) := \frac{\inf\{\|u - \hat{u}\|_{H^\gamma (M)} : \hat{u} \in \mathcal{M}_1 \cap \mathcal{B}(u_* ,\delta) \} }{\|u\|_{H^\gamma (M)}}.
\end{align*}
\begin{proposition}[Local stability]\label{Proposition:localstab}
    Let $(M^n , [\hat{h}])$ be the conformal infinity of an asymptotically hyperbolic manifold $(X^{n+1},g^+ )$ such that $\Lambda_{\gamma} (M^n ,[\hat{h}]) < \Lambda_{\gamma} (\mathbb{S}^n ,[h_{0}])$. Then, there exist constants $C>0$, $\delta>0$ and $\theta \geq 0$ such that
    \begin{align*}
        I_{\gamma} [u,\hat{h}]  - \Lambda_{\gamma} (M,[\hat{h}]) \geq C d_{\delta}(u,\mathcal{M}_1)^{2+\theta} \, ,
    \end{align*}
    for all $u \in \mathcal{B}(u_*,\delta)$. If $u_*$ is non-degenerate or integrable, then we may take $\theta =0$.
\end{proposition}
\begin{proof}[Proof of Proposition \ref{Proposition:localstab}]
    Given a fixed minimizer $u_* \in \mathcal{M}_1$, we denote by $F : \mathcal{U} \to K^\perp$ the Lyapunov-Schmidt reduction map adapted to $u_*$ as in Lemma (\ref{Lemma:LS}), where $K$ is the kernel of $\nabla_{\mathcal{B}}^2 I_\gamma [u_*]$. Thanks to Lemma \ref{Lemma:LS}(\ref{LS2}), for any $u \in \mathcal{B}(u_* ,\delta)$, we can define the Lyapunov-Schmidt projection onto the critical manifold $\mathcal{S}$
    \begin{align*}
        P(u) = u_* + \pi_K (u - u_*) + F(\pi_K (u - u_*)) \in \mathcal{S}.
    \end{align*}
    Since $P(u_*)=u_*$ and $P$ is continuous, after making $\delta$ smaller if necessary, we may assume that
    \begin{align*}
        \|P(u) - u_*\|_{H^\gamma (M)}<\frac{\varepsilon}{2},
            \qquad
        \|u - P(u)\|_{H^\gamma (M) }< \frac{\varepsilon}{2}
    \end{align*}
    for every \(u\in\mathcal B(u_*,\delta)\). Now, to estimate the deficit, we split into two terms
    \begin{align*}
        I_\gamma [u] - \Lambda_\gamma (M,[\hat{h}]) = I_\gamma [u] - I_\gamma [P(u)] + I_\gamma [P(u)] - \Lambda_\gamma (M,[\hat{h}]),
    \end{align*}
    and estimate them separately.

    For the first term, we perform a Taylor expansion of $I_\gamma$ around $P(u)$ to get
    \begin{align*}
        I_\gamma [u] - I_\gamma [P(u)] =& \nabla_{ \mathcal{B} } I_\gamma [P(u)] [u - P(u)] + \frac{1}{2}\nabla_{\mathcal{B} }^2 I_\gamma [\zeta] [u - P(u),u - P(u)] \\
        =& \frac{1}{2}\nabla_{\mathcal{B}}^2 I_\gamma [\zeta] [u - P(u),u - P(u)] \\
        =& \frac{1}{2}\nabla_{\mathcal{B}}^2 I_\gamma [u_*] [u - P(u),u - P(u)] + \|u - P(u)\|^2 o(1) \, ,
    \end{align*}
    where $\zeta$ is a geodesic in $\mathcal{B}$ joining $u$ and $P(u)$. The second equality holds since $(u-P(u)) \in K^\perp$, and
    \begin{align*}
        \pi_{K^\perp} \nabla_{\mathcal{B}}I_\gamma [P(u)] = 0 \, .
    \end{align*}
    Meanwhile, the last equality comes from the uniform modulus of continuity of $\nabla_{\mathcal{B}}^2 I_\gamma [-]$ and the properties given in \eqref{lemma:Variationalppties}. More explicitly, by (\ref{lemma:variat_2}), we get
    \begin{align*}
        \frac{1}{2}\nabla_{\mathcal{B}}^2 I_\gamma [\zeta] = \frac{1}{2}\nabla_{\mathcal{B}}^2 I_\gamma [u_*] + o(1),
    \end{align*}
    since $\|\zeta - u_*\| < \varepsilon$. Because $u_*$ is a minimizer, the quadratic form associated to
    $\nabla_{\mathcal B}^2 I_\gamma[u_*]$ is nonnegative. Moreover, $\mathcal{L}_{u_*}$ is a self-adjoint elliptic pseudodifferential operator
    of order $2\gamma$ on the compact manifold $M$, and therefore it has discrete spectrum. Hence, its restriction to $K^\perp$ has a positive spectral gap. So, noting that $(u - P(u))\in K^\bot$, we take the first positive eigenvalue $\lambda_1>0$ to get the coercivity bound
    \begin{align*}
        I_\gamma [u] - I_\gamma [P(u)] =& \frac{1}{2}\nabla_{\mathcal{B}}^2 I_\gamma [u_*] [u - P(u),u - P(u)] + \|u - P(u)\|_{H^\gamma (M)}^2 o(1) \\
        \geq& \frac{1}{2}\lambda_1 \|u - P(u)\|_{H^\gamma (M)}^2 + \|u - P(u)\|_{H^\gamma (M)}^2 o(1)\\
        \geq& C \|u - P(u)\|_{H^\gamma (M)}^2 .
    \end{align*}
    Note that this constant $C$ depends on the minimizer $u_*$.

    \medskip

    For the second term we recall that if $u_*$ is nondegenerate, $P(u) = u_*$ and hence the term vanishes. If $u_*$ is degenerate but integrable, $I_\gamma$ is locally constant around $u_*$ by Lemma \ref{lemma:integrable}, hence $I_\gamma [P(u)] = I_\gamma [u_*] $.

    \medskip
    
    However, for the non-integrable case we work a little bit more using the
\L ojasiewicz inequality. Indeed, bounding the second term and writing
$\phi:=\pi_K(u-u_*)$, we obtain
\begin{align*}
    I_\gamma[P(u)]-\Lambda_\gamma(M,[\hat h])
    &=
    I_\gamma[P(u)]-I_\gamma[u_*]\\
    &=
    q(\phi)-q(0)\\
    &\geq
    C\inf\left\{
    |\phi-\hat\phi|:\,
    \hat\phi\in B(0,\delta)\cap K,\,
    q(\hat\phi)=q(0)
    \right\}^{2+\theta}.
\end{align*}
We now claim that, since $F$ satisfies the Lipschitz estimate
\eqref{LS3}, the finite-dimensional distance to the minimum level set of
$q$ controls the infinite-dimensional distance from $P(u)$ to the nearby
minimizers. More precisely,
\begin{align}\label{eq:distancecomparison}
    \inf \big\{|\phi-\hat{\phi}| :\,
    &\hat{\phi}\in B(0,\delta)\cap K,\,
    q(\hat{\phi})=q(0)\big\} \notag\\
    &\geq C
    \inf \big\{
    \|P(u)-\hat{u}\|_{H^\gamma(M)}:\,
    \hat{u}\in\mathcal{M}_1\cap\mathcal{B}(u_*,\delta)
    \big\}.
\end{align}
Indeed, for every $\hat{\phi}\in B(0,\delta)\cap K$ such that
$q(\hat{\phi})=q(0)$, let
\[
    \hat{u}:=u_*+\hat{\phi}+F(\hat{\phi}).
\]
Since $\hat{u}\in\mathcal{S}\subset\mathcal{B}$ and
\[
    I_\gamma[\hat{u}]
    =q(\hat{\phi})
    =q(0)
    =I_\gamma[u_*]
    =\Lambda_\gamma(M,[\hat h]),
\]
we have $\hat{u}\in\mathcal{M}_1$. After decreasing $\delta$ if necessary,
we may also assume that $\hat{u}\in\mathcal{B}(u_*,\delta)$. Since
$P(u)=u_*+\phi+F(\phi)$, we obtain
\begin{align*}
    \|P(u)-\hat{u}\|_{H^\gamma(M)}
    &\leq
    \|\phi-\hat{\phi}\|_{H^\gamma(M)}
    +
    \|F(\phi)-F(\hat{\phi})\|_{H^\gamma(M)}
    \\
    &\leq
    \|\phi-\hat{\phi}\|_{H^\gamma(M)}
    +
    C\|\phi-\hat{\phi}\|_{C^\alpha(M)}
    \\
    &\leq
    C|\phi-\hat{\phi}|,
\end{align*}
where in the last inequality we used the equivalence of norms on the
finite-dimensional space $K$. Taking the infimum over all
$\hat{\phi}\in B(0,\delta)\cap K$ satisfying $q(\hat{\phi})=q(0)$ proves
\eqref{eq:distancecomparison}. Consequently, the
    finite-dimensional \L ojasiewicz inequality gives
    \begin{align}\label{eq:reduced-distance}
        I_\gamma [P(u)]-\Lambda_\gamma(M,[\hat{h}])
        \geq C
        \inf\big\{
        \|P(u)-\hat{u}\|_{H^\gamma (M)}:\,
        \hat{u}\in\mathcal{M}_1 \cap \mathcal{B}(u_*,\delta)
        \big\}^{2+\theta}.
    \end{align}

    On the other hand, the estimate for the transverse component gives
    \begin{align}\label{eq:transverse-distance}
        I_\gamma [u]-I_\gamma [P(u)]
        \geq C\|u - P(u)\|_{H^\gamma (M)}^2.
    \end{align}
    After decreasing $\delta$ if necessary, we may assume that
    $\|u - P(u)\|_{H^\gamma (M)}\leq 1$. Since $2+\theta \geq 2$, it follows that
    \begin{align*}
        \|u - P(u)\|_{H^\gamma (M)}^2
        \geq
        \|u - P(u)\|_{H^\gamma (M)}^{2+\theta}.
    \end{align*}
    Combining \eqref{eq:reduced-distance} and
    \eqref{eq:transverse-distance}, we obtain
    \begin{align*}
        I_\gamma [u]-\Lambda_\gamma(M,[\hat{h}])
        \geq C\Big(
        &\|u - P(u)\|_{H^\gamma (M)}^{2+\theta}
        \\
        &+
        \inf\big\{
        \|P(u)-\hat{u}\|_{H^\gamma (M)}:\,
        \hat{u}\in\mathcal{M}_1 \cap \mathcal{B}(u_*,\delta)
        \big\}^{2+\theta}
        \Big).
    \end{align*}
    Finally, by the triangle inequality,
    \begin{align*}
        \inf\big\{
        \|u - \hat{u}\|_{H^\gamma (M)}:\,
        \hat{u}\in\mathcal{M}_1 \cap \mathcal{B}(u_*,\delta)
        \big\}
        \leq{}&
        \|u - P(u)\|_{H^\gamma (M)}
        \\
        &+
        \inf\big\{
        \|P(u)-\hat{u}\|_{H^\gamma (M)}:\,
        \hat{u}\in\mathcal{M}_1 \cap \mathcal{B}(u_*,\delta)
        \big\}.
    \end{align*}
    Using $(a+b)^{2+\theta}\leq
    C(a^{2+\theta}+b^{2+\theta})$, we conclude that
    \begin{align*}
        I_\gamma [u]-\Lambda_\gamma(M,[\hat{h}])
        \geq C
        \inf\big\{
        \|u - \hat{u}\|_{H^\gamma (M)}:\,
        \hat{u}\in\mathcal{M}_1 \cap \mathcal{B}(u_*,\delta)
        \big\}^{2+\theta}.
    \end{align*}
\end{proof}

\subsection{Global quantitative stability}
We now combine the local stability estimates obtained in the previous subsection with the compactness of minimizing sequences to establish the global quantitative stability result. The key point is that sufficiently small energy deficit forces a function to lie in a neighborhood of the set of minimizers, where the local analysis applies.

The required compactness statement is proved in Appendix \ref{appendix:B}. Its proof relies on concentration-compactness arguments due to P.-L. Lions \cite{Lions1,Lions2}, see also \cite{Valdinocci}, together with an \emph{almost} weighted Sobolev trace inequality established by Jin and Xiong \cite{JinXiong}. We just state the result for the sake of completeness of the section.

\begin{lemma}[Compactness of minimizing sequences]\label{lemma:compacntess}
Let $(M^n,[\hat{h}])$ be the conformal infinity of an asymptotically hyperbolic manifold $(X^{n+1},g^+)$ satisfying $\Lambda_{\gamma}(M,[\hat{h}])<
\Lambda_{\gamma}(\mathbb{S}^n,[h_0]).$ Let $(w_k)_{k\in\mathbb N}\subset\mathcal{B}$ be a minimizing sequence, namely $I_\gamma[w_k,\hat{h}] \to \Lambda_\gamma(M,[\hat{h}])$. Then, up to a subsequence, $w_k\to w_*$ strongly in $H^\gamma(M)$, for some minimizer $w_* \in \mathcal{M}_1$.
\end{lemma}

We now prove the global stability result, Theorem \ref{thm:mainv1}.
\begin{proof}[Proof of Theorem \ref{thm:mainv1}]
    This proof is essentially turning the local estimates obtained in Proposition \ref{Proposition:localstab} into global using the compactness of minimizing sequences, see Lemma \ref{lemma:compacntess}. Indeed, fix a minimizer $w_* \in \mathcal{M}_1$, and let $\delta>0$, $\theta\geq 0$ and $C>0$ be the constants given by Proposition \ref{Proposition:localstab}. By the compactness result, $\mathcal{M}_1$ is compact in $H^\gamma (M)$ hence we may cover by balls $\mathcal{B}(w_*, \delta / 4)$ and extract a finite cover $\{\mathcal{B}(w_i ,\delta (w_i) / 4)\}_{i\in I}$. Then define
    \begin{align*}
        &\delta_0 := \min_{i\in I} \delta(w_i)/2 >0 \, ,\\
        &\theta_0 := \max_{i\in I} \theta(w_i) < \infty \, ,\\
        &C_0 := \min_{i\in I} C(w_i) >0 \, .
    \end{align*}
    Now, since $\mathcal{M}_1 \subset \mathcal{M}$ and the inequality desired in \ref{thm:mainv1} is zero-homogeneous in $w$ we can work for $w\in \mathcal{B}$ without losing generality. Let $w\in \mathcal{B}$ be such that $d(w,\mathcal{M}_1) < \delta_0 /4 $. Then, we can find an $i\in I$ such that $\|w - w_i\|_{H^\gamma} < \delta_i /2$. By the triangle inequality, if $\tilde{w} $ is the closest element of $\mathcal{M}_1$ to $w$, we have $\|\tilde{w} - w_i\| < \delta_i$. Thus we apply Proposition \ref{Proposition:localstab} to obtain
    \begin{align*}
        I_\gamma [w,\hat{h}] - \Lambda_\gamma (M,[\hat{h}]) \geq c(w_i) d_{\delta_i }(w,\mathcal{M}_1)^{2+ \theta_i} \geq C_0 d(w,\mathcal{M}_1)^{2+ \theta_0}\, , 
    \end{align*}
    which is the desired result.

    \medskip

    We are left with the case $d(w,\mathcal{M}_1) >\delta_0 /4$. Using the fact that $w $ has unit norm we obtain, after applying the triangle inequality and a normalization argument, that $d(w,\mathcal{M})>\delta_0 /16$. By Lemma \ref{lemma:compacntess}, there exists $\varepsilon>0$ such that
    \begin{align*}
        I_\gamma [w,\hat{h}]-\Lambda_\gamma (M,[\hat{h}]) < \varepsilon \Longrightarrow d(w,\mathcal{M}) < \delta_0 /16 \, .
    \end{align*}
    Hence, since $d(w,\mathcal{M})>\delta_0 /16$ we have $I_\gamma [w,\hat{h}]-\Lambda_\gamma (M,[\hat{h}]) > \varepsilon$. Since $d(w,\mathcal{M}) \leq 1$, letting $C := \min \{C_0 ,\varepsilon\}$, we have proved the stability for all $w \in \mathcal{B}$.
\end{proof}

Using the equivalence given by \cite{ChangGonzalez}, we immediately have an analogous stability result for the extension problem. For that we define, for any $U\in W^{1,2}(X,y^a)$, the extension distance
\begin{align*}
    d(U,\mathcal{M}^*) := \frac{\inf\{\|U - U_*\|_{W^{1,2}(X,y^a)} : U_* \in \mathcal{M}^* \}}{\|U\|_{W^{1,2}(X,y^a)}},
\end{align*}
where $\mathcal{M}^*$ is the set of minimizers of $I_\gamma^*[-,\bar{g}^*]$. We now establish the corollary for the extension.
\begin{corollary}[Quantitative Stability of $I_\gamma^*$] Let $(M^n , [\hat{h}])$ be the conformal infinity of an asymptotically hyperbolic manifold $(X^{n+1},g^+ )$ such that $\Lambda_{\gamma} (M,[\hat{h}]) < \Lambda_{\gamma} (\mathbb{S}^n ,[h_{0}])$. Then, there exist constants $C>0$ and $\theta \geq 0$ such that
\begin{align*}
    I_\gamma^* [U,\bar{g}^*] - \Lambda_\gamma (M,[\hat{h}]) \geq C d(U, \mathcal{M}^*)^{2+\theta},
\end{align*}
for all $U \in W^{1,2}(X,y^a)$ with nonzero trace.
\end{corollary}
\begin{proof}
    Let $U \in W^{1,2}(X,y^a)$, then we decompose it as $U = U_w + U_0$, where $T(U_w)=w $, for some $w \in H^\gamma (M)$, is the weighted harmonic extension and $TU_0 = 0$. We estimate
    \begin{align*}
        I_\gamma^* [U,\bar{g}^*] = I_\gamma^* [U_w + U_0 ,\bar{g}^*] &= \frac{d_\gamma^*\int_X y^a |\nabla U_w + \nabla U_0|^2 dv_{\bar{g}^*} + \int_M Q_\gamma^{\hat{h}} w^2 d\sigma_{\hat{h}} }{\left(\int_M |w|^{2^*} d\sigma_{\hat{h}} \right)^{\frac{2}{2^*}}}  \\
        &= I_\gamma^* [U_w, \bar{g}^*] + d_\gamma^*  \frac{2\int_X y^a \langle \nabla U_w , \nabla U_0  \rangle dv_{\bar{g}^*} + \int_X y^a |\nabla U_0|^2 dv_{\bar{g}^*}  }{\left(\int_M |w|^{2^*} d\sigma_{\hat{h}} \right)^{\frac{2}{2^*}}} \\
        &= I_\gamma [w,\hat{h}] + \frac{d_\gamma^* \int_X y^a |\nabla U_0|^2 dv_{\bar{g}^*}  }{\left(\int_M |w|^{2^*} d\sigma_{\hat{h}} \right)^{\frac{2}{2^*}}} ,
    \end{align*}
    where we used that $\nabla U_w$ and $\nabla U_0$ are orthogonal, see the proof of Proposition \ref{Propositio:Stabilitytrace}. Hence, when estimating the deficit we apply Theorem \ref{thm:mainv1}
    \begin{align*}
        I_\gamma^* [U,\bar{g}^*] - \Lambda_\gamma (M,[\hat{h}]) &= I_\gamma [w,\hat{h}] + \frac{d_\gamma^* \int_X y^a |\nabla U_0|^2 dv_{\bar{g}^*}  }{\left(\int_M |w|^{2^*} d\sigma_{\hat{h}} \right)^{\frac{2}{2^*}}} -\Lambda_\gamma (M,[\hat{h}])\\
        \geq& C \left( \frac{\inf\{ \|w - u\|_{H^\gamma (M)} : u \in \mathcal{M}  \} }{\|w\|_{H^\gamma (M)} }  \right)^{2+\theta}
        + \frac{d_\gamma^* \int_X y^a |\nabla U_0|^2 dv_{\bar{g}^*}  }{\left(\int_M |w|^{2^*} d\sigma_{\hat{h}} \right)^{\frac{2}{2^*}}} \\
        \geq& C\left( \frac{\inf\{ \|w - u\|_{H^\gamma (M)} : u \in \mathcal{M}  \} }{\|w\|_{H^\gamma (M)} }  \right)^{2+\theta}
        + \frac{d_\gamma^* \int_X y^a |\nabla U_0|^2 dv_{\bar{g}^*}  }{\|w\|_{H^\gamma (M)}^2} ,
    \end{align*}
    where we applied the Sobolev embedding. Now, noting that $U_0$ has zero trace, we derive that it satisfies a Poincaré inequality for the weighted spaces. With this we bound
    \begin{align*}
        I_\gamma^* [U,\bar{g}^*] - \Lambda_\gamma (M,[\hat{h}]) &\geq C\left( \frac{\inf\{ \|w - u\|_{H^\gamma (M)} : u \in \mathcal{M}  \} }{\|w\|_{H^\gamma (M)} }  \right)^{2+\theta}
        + C_2 \frac{ \|U_0\|_{W^{1,2}(X,y^a)}^2 }{\|w\|_{H^\gamma (M)}^2} .
    \end{align*}
    Using the continuity of the trace and harmonic extension, we obtain
\begin{align*}
    \|w\|_{H^\gamma(M)}
    &\leq
    C\|U\|_{W^{1,2}(X,y^a)},\\
    \|U_0\|_{W^{1,2}(X,y^a)}
    &=
    \|U-U_w\|_{W^{1,2}(X,y^a)}
    \leq
    C\|U\|_{W^{1,2}(X,y^a)}.
\end{align*}
Moreover, by the continuity of the harmonic extension operator,
\begin{align*}
    \inf\{\|U-U_*\|_{W^{1,2}(X,y^a)}:U_*\in\mathcal M^*\}
    \leq
    C\inf\{\|w-u_*\|_{H^\gamma(M)}:u_*\in\mathcal M\}
    +
    \|U_0\|_{W^{1,2}(X,y^a)}.
\end{align*}
Hence, setting
\begin{align*}
    a:=
    \frac{\inf\{\|w-u_*\|_{H^\gamma(M)}:u_*\in\mathcal M\}}
    {\|w\|_{H^\gamma(M)}},
    \qquad
    b:=
    \frac{\|U_0\|_{W^{1,2}(X,y^a)}}
    {\|w\|_{H^\gamma(M)}},
\end{align*}
we deduce
\begin{align*}
    d(U,\mathcal M^*)
    \leq
    C\frac{a+b}{1+b}.
\end{align*}
Finally, since $\mathcal M$ is invariant under multiplication by positive constants, we have $a\leq 1$. Therefore,
\begin{align*}
    d(U,\mathcal M^*)^{2+\theta}
    \leq
    C(a^{2+\theta}+b^2),
\end{align*}
and the conclusion follows from the previous estimate.
\end{proof}

\subsection{Proof of Corollary \ref{Cor:conformal}} As expected, Corollary \ref{Cor:conformal} is a direct consequence of Theorem \ref{thm:mainv1}. We omit most of the calculations since they are completely analogous to the ones done for Corollary 1.2 in \cite{Neumayer}. Generally speaking, the proof relies heavily on the conformal invariance of the, in this case, pseudo-differential operator $P_\gamma$, see (\ref{eq:GJMSconformal}).

\begin{proof}[Proof of Corollary \ref{Cor:conformal}] To show (\ref{eq:conformal1}) we apply directly Theorem \ref{thm:mainv1} and the fractional Sobolev inequality on $(M,[\hat{h}])$. Plugging $\tilde{h} = u^{4/(n-2\gamma)} \hat{h}$ we obtain
    \begin{align*}
        \mathcal{I}_\gamma (\tilde{h}) - \Lambda_\gamma (M,[\hat{h}]) = I_\gamma [u,\hat{h}] - \Lambda_\gamma (M,[\hat{h}]) \geq& C \left( \frac{\inf\{ \|u - w\|_{H^\gamma (M)} : w \in \mathcal{M}  \}  }{\|u\|_{H^\gamma (M)}}\right)^{2+\theta}\\
        \geq& C \left( \frac{\inf\{ \|u - w\|_{L^{2^*} (M)} : w \in \mathcal{M}  \}  }{\|u\|_{L^{2^*} (M)}}\right)^{2+\theta}\\
        =& C\left( \frac{\inf\{ \|\tilde{h} - h\| : h \in \mathcal{M}  \}  }{\operatorname{vol}_{\tilde{h}} (M)^{1/2^*}  }\right)^{2+\theta} \, .
    \end{align*}
To see the second inequality above we split into when $u$ is close to the set of minimizers, and when it's not. If $d(u,\mathcal{M}) \leq \delta$, both denominators are comparable and the estimate holds. On the other hand, if $d(u,\mathcal{M}) > \delta$ we note that $\inf_{w\in \mathcal{M}} \|u-w\|_{L^{2^*} (M)}/ \|u\|_{L^{2^*} (M)} \leq 1$, so the estimate follows from choosing $C$ small enough.

Now we show the estimate (\ref{eq:conformal2}). Indeed, since $\Lambda_\gamma (M,[\hat{h}]) >0$, the quadratic form associated to $P_\gamma^{\hat{h}}$ is positive definite and, by ellipticity, the induced norm is equivalent to the original one in $H^\gamma (M)$. To conclude, it's enough to show that $\|-\|_*$ is independent of the choice of conformal representative. Let $h,\hat{h} \in \mathcal{M}_1$, with $h = \phi^{4/(n-2\gamma)}\hat{h}$. Then, $h_u = u^{4/(n-2\gamma)} h = (\phi u)^{4/(n-2\gamma)} \hat{h} = \hat{h}_{\phi u}$, and the same for $h_v = \hat{h}_{\phi v} $. We have
\begin{align*}
    \|h_u - h_v\|_*^2 =& \int_M (u-v) P_\gamma^{h} (u-v) d\sigma_h\\
                      =& \int_M (u-v) P_\gamma [\phi^{\frac{4}{n-2\gamma}} \hat{h} ](u-v) \phi^{\frac{2n}{n-2\gamma}} d\sigma_{\hat{h}}\\
                      =& \int_M (u-v) \phi^{-\frac{n+2\gamma}{n-2\gamma}} P_\gamma^{\hat{h}} (\phi(u-v)) \phi^{\frac{2n}{n-2\gamma}} d\sigma_{\hat{h}}\\
                      =& \int_M (\phi u - \phi v)  P_\gamma^{\hat{h}} (\phi u - \phi v) d\sigma_{\hat{h}}\\
                      =& \|\hat{h}_{\phi u} - \hat{h}_{\phi v}\|_*^2\, ,
\end{align*}
where we used the conformal covariance of $P_\gamma^{\hat{h}}$, see (\ref{eq:GJMSconformal}). Hence, $\|-\|_*$ is independent of conformal choice and the result is proved.
\end{proof}

\subsection{Proof of Proposition \ref{Proposition:AS}} Although the fractional Yamabe problem is governed by a nonlocal operator, once the Lyapunov--Schmidt reduction has been established the remaining analysis becomes entirely finite dimensional. Indeed, the reduced functional $q$ is a real analytic function on the finite-dimensional kernel $K$, and all the notions of integrability, order of integrability and the Adams--Simon positivity condition are encoded in its Taylor expansion. Consequently, the proof below follows the same finite-dimensional argument as in the classical Yamabe problem, the only difference lying in the construction of the Lyapunov--Schmidt reduction.

The notion of integrability first appeared in the study of the classical Yamabe flow and has since played an important role in the analysis of degenerate critical points. To the best of our knowledge, it has not been explicitly introduced in the fractional Yamabe setting. Nevertheless, once the Lyapunov--Schmidt reduction is available, the definition and its consequences are completely analogous to the classical case, since they depend only on the finite-dimensional reduced functional.

We now introduce the setting. Let $u_0 \in \mathcal{M}_1$ be a nonintegrable minimizer. Let $q:\mathcal{U} \to \mathbb{R}$, where $\mathcal{U} \subset K$, be the function defined in Lemma \ref{Lemma:LS}. Since it is analytic we can expand in power series
\begin{align*}
    q(x) = q(0) + \sum_{j\geq p} q_j (x),
\end{align*}
where each $q_j$ is a degree $j$ homogeneous polynomial and $p$ is the least integer so that $q_p  \neq 0$. We call this $p$ the \emph{order of integrability} of $u_0$. We note that there must be that $p\geq 3$, since on $K$, $\nabla q (0) = \nabla^2 q(0) = 0$.

Next, we define the Adams-Simon positivity condition.
\begin{definition}[$\text{AS}_p$ condition]\label{def:AS}
    We say $u_0$ satisfies the Adams-Simon positivity condition of order $p$, $\text{AS}_p$ in short, if $p$ is the order of integrability of $u_0$ and $q_p|_{\mathbb{S}^{l-1}}$ attains a positive maximum for some $v \in \mathbb{S}^{l-1}$. Here $\mathbb{S}^{l-1}$ represents the unit sphere with respect to the $L^{2}(M)$ norm in $K \subset H^\gamma (M)$.
\end{definition}

We now finish with the proof of Proposition \ref{Proposition:AS}, which is an adaptation of the one for Proposition 4.3 in \cite{Neumayer}.
\begin{proof}[Proof of Proposition \ref{Proposition:AS}] Let $v \in \mathbb{S}^{l-1} $ be the maximum of $q_p$. For small $t$ we define $u_t = u_0 + tv + F(tv) $, where $F$ is the map defined in Lemma \ref{Lemma:LS}. Using the properties of $F$ and the equivalence of norms in $K$ we have
\begin{align}\label{eq:PropAS}
    \|u_t - u_0\|_{H^\gamma (M)} \sim t.
\end{align}
By definition of $q$ we have
\begin{align*}
    I_\gamma [u_t] - I_\gamma [u_0] = q(tv) - q(0) = \sum_{j\geq p} t^j q_j (v) = t^p q_p (v) + O(t^{p+1}).
\end{align*}
Since $u_0 $ satisfies $\text{AS}_p$ we conclude
\begin{align*}
    |I_\gamma [u_t] - I_\gamma [u_0]| \leq Ct^p q_p (v),
\end{align*}
for $t$ sufficiently small. This and the estimate (\ref{eq:PropAS}) yield the desired conclusion.
\end{proof}

\appendix

\section{Proof of Lemma \ref{Lemma:LS}}\label{Appendix:A}

We follow the proof of the classical Yamabe Lyapunov--Schmidt reduction done in \cite{Neumayer}. The main difference is that now the linearized operator is a nonlocal elliptic operator on a closed manifold.

\begin{proof}[Proof of Lemma \ref{Lemma:LS}]
Fix $u \in \mathcal{M}_1$, and let $K$ and $K^\bot$ be as in Lemma \ref{Lemma:LS}. We divide the proof into steps.
\medskip

\noindent
\emph{Step 1}: Define the map $F$. Define the map $\mathcal{N}: C^{2\gamma +\alpha}(M) \cap \mathcal{B} \to C^{\alpha}(M) \cap T_u \mathcal{B}$ by
\[
\mathcal N(w)
=
\pi_K(w-u)
+
\pi_{K^\perp}\nabla_{\mathcal B}I_\gamma (w),
\]
Then, clearly $\mathcal{N}(u) = 0$. If we consider a smooth curve $(w_t)_{t\in (-\delta , \delta)} \subset C^{2\gamma +\alpha}(M) \cap \mathcal B$ with $w_0 = u$ and $\partial_t |_{t=0} w_t = \xi$, then
\begin{align*}
    \nabla_{\mathcal{B}} \mathcal{N} (u)[\xi] = \frac{d}{dt}|_{t=0} \, \mathcal{N} (w_t ) = \pi_K \xi + \pi_{K^\perp}
\nabla_{\mathcal B}^2I_\gamma (u)[\xi,-].
\end{align*}
Since the Hessian is self-adjoint, its image is contained in $K^\perp $. Indeed, if $\phi \in K$, then
\begin{align*}
    \left\langle
\nabla_{\mathcal B}^2I_\gamma(u)[\xi,-],
\phi
\right\rangle
=
\nabla_{\mathcal B}^2 I_\gamma (u)[\xi,\phi]
=
\nabla_{\mathcal B}^2 I_\gamma(u)[\phi,\xi]
=
0.
\end{align*}
Thus
\begin{align*}
    \nabla_{\mathcal{B}} \mathcal{N} (u)[\xi] = \pi_K \xi + \nabla_{\mathcal B}^2 I_\gamma (u)[\xi,-].
\end{align*}
We claim that $\nabla_{\mathcal{B}} \mathcal{N} (u)$ is an isomorphism. Indeed, if $\nabla_{\mathcal{B}} \mathcal{N} (u)[\xi] = 0$ then
\begin{align*}
    \pi_K\xi=0\, ,
\end{align*}
and
\begin{align*}
    \nabla_{\mathcal B}^2I_\gamma(u)[\xi,-]=0.
\end{align*}
Hence, $\xi\in K^\perp\cap K$ so $\xi =0$. Surjectivity follows from the Fredholm alternative. Indeed, $\nabla_{\mathcal B}^2I_\gamma(u)\big|_{K^\perp} : K^\perp\to K^\perp$ is a Fredholm operator of index zero and has trivial kernel. Hence, it is an isomorphism. Therefore, by the inverse function theorem, $\mathcal{N}$ has a local inverse $\mathcal{N}^{-1}$ defined on a neighborhood $\widehat{\mathcal{U}}$ of $0$ in $C^{\alpha}(M) \cap T_u \mathcal{B}$. Set $\mathcal{U} := K\cap \widehat{\mathcal{U}} \subset K$ and define $F:\mathcal{U} \to K^\perp$
\begin{align}\label{eq:LS_F}
    F(\phi) = \pi_{K^\perp} \left(\mathcal N^{-1}(\phi) - u \right).
\end{align}

\medskip

\noindent
\emph{Step 2: Basic observations of the map $F$.} We now prove some basic identities of $F$. For every $\phi \in \mathcal{U}$, using the definition of \(\mathcal N\), we have
\begin{align}\label{eq:LS_Phi}
    \phi = \mathcal N(\mathcal N^{-1}(\phi)) = \pi_K(\mathcal N^{-1}(\phi) - u) +
\pi_{K^\perp} \nabla_{\mathcal B}I_\gamma  (\mathcal N^{-1}(\phi)).
\end{align}
Taking $\pi_K$ to equation (\ref{eq:LS_Phi}) gives
\begin{align*}
    \phi = \pi_K(\mathcal N^{-1}(\phi) - u).
\end{align*}
Together with the definition of $F$, this implies
\begin{align}\label{eq:LS_invN}
    \mathcal N^{-1}(\phi) = u + \phi + F(\phi).
\end{align}
Taking instead $\pi_{K^\perp}$ to equation (\ref{eq:LS_Phi}) yields
\begin{align}\label{eq:LS_perpe}
    \pi_{K^\perp} \nabla_{\mathcal B}I_\gamma (u + \phi + F(\phi)) = 0.
\end{align}
Therefore
\begin{align*}
    \nabla_{\mathcal B}I_\gamma
(u + \phi + F(\phi)) = \pi_K
\nabla_{\mathcal B}I_\gamma (u + \phi + F(\phi)).
\end{align*}

Next, differentiating equation (\ref{eq:LS_invN}) in the direction $\psi\in K$, we obtain
\begin{align*}
    \nabla \mathcal N^{-1}(\phi)[\psi] = \psi + \nabla F(\phi)[\psi].
\end{align*}
Thus
\begin{align}\label{eq:LS_nablaF}
    \pi_K \nabla \mathcal{N}^{-1}(\phi)[\psi] = \psi,
\qquad
\pi_{K^\perp} \nabla \mathcal{N}^{-1}(\phi)[\psi]
=
\nabla F(\phi)[\psi].
\end{align}
Since $\mathcal N(u)=0$ and $\nabla \mathcal{N} (u)|_K=\operatorname{Id}_K$, using the definition of $F$ (see (\ref{eq:LS_F})), we get $F(0)=0$ and $\nabla F(0)=0$.

\medskip
\noindent
\emph{Step 3: Check properties of $F$.} We start checking property (\ref{LS1}). Indeed, by the construction above (see (\ref{eq:LS_invN})), for all $\phi \in \mathcal{U}$
\begin{align*}
    u + \phi + F(\phi) = \mathcal{N}^{-1}(\phi) \in \mathcal{B}.
\end{align*}
Hence
\begin{align*}
    \mathcal{S} = \left\{ u + \phi + F(\phi): \phi\in \mathcal{U} \right\} \subset \mathcal{B}.
\end{align*}
Moreover, by (\ref{eq:LS_perpe})
\begin{align*}
    \nabla_{\mathcal B}I_\gamma
(u + \phi + F(\phi)) = \pi_K \nabla_{\mathcal B}I_\gamma (u + \phi + F(\phi)).
\end{align*}
Now define $q:\mathcal{U} \to \mathbb{R}$ by $q(\phi) := I_\gamma (u + \phi + F(\phi))$. For every $\psi \in K$, by the chain rule, we have
\begin{align*}
    \frac{d}{dt} \bigg|_{t=0} q( \phi + t\psi) =
\nabla_{\mathcal{B}} I_\gamma (u + \phi + F(\phi))
[
\psi + \nabla F(\phi)[\psi]
].
\end{align*}
Since $\nabla F(\phi)[\psi]\in K^\perp$ and $\nabla_{\mathcal B} I_\gamma
(u + \phi + F(\phi))\in K$, the $K^\perp$-component vanishes, and therefore
\begin{align*}
    \frac{d}{dt}\bigg|_{t=0} q(\phi + t\psi) =
\nabla_{\mathcal{B}} I_\gamma (u + \phi + F(\phi))[\psi].
\end{align*}
Thus
\begin{align*}
    \nabla q(\phi) = \pi_K \nabla_{\mathcal B} I_\gamma
(u + \phi + F(\phi)).
\end{align*}
This proves all the remaining assertions in (\ref{LS1}), except for the analyticity of the reduced functional. For this, we argue as in \cite[Lemma 4.3]{ChanFlow}. After shrinking the neighborhood if necessary, all functions under consideration remain strictly positive. Hence, the nonlinear maps appearing in $\mathcal{N}$ are real analytic in $C^{2\gamma+\alpha}(M)$. Since $P_\gamma^{\hat h}$ is linear, the map $\mathcal{N}$ is real analytic. Therefore, the analytic inverse function theorem implies that $\mathcal{N}^{-1}$, and consequently $F$, are real analytic. It follows that the reduced functional
\begin{align*}
    q(\phi)=I_\gamma(u+\phi+F(\phi))
\end{align*}
is real analytic, completing the proof of (\ref{LS1}).

\medskip

Now we show property (\ref{LS2}). Since \(K\) is finite-dimensional, all norms are equivalent on \(K\). Thus, there exists $\varepsilon >0 $ such that
\begin{align*}
    \{ \phi \in K: \|\phi\|_{H^\gamma (M)} < \varepsilon\} \subset \mathcal{U}.
\end{align*}
Because the projection $\pi_K$ is bounded, we may choose $0< \delta \leq \varepsilon$ such that $\|w - u\|_{H^\gamma (M)} < \delta$ implies $\pi_K (w-u) \in \mathcal{U}$. Let now $w \in \mathcal{B}(u , \delta) $ be a critical point, so we have $\nabla_{\mathcal{B}} I_\gamma (w)=0$. Therefore
\begin{align*}
    \mathcal{N} (w) = \pi_K (w - u).
\end{align*}
Since $w$ lies in the neighborhood on which $\mathcal{N}$ is invertible, we get
\begin{align*}
    w = \mathcal{N}^{-1}(\mathcal{N}(w)) = \mathcal{N}^{-1}(\pi_K (w - u)).
\end{align*}
Using the identity (\ref{eq:LS_invN}) for $\mathcal{N}^{-1}$ yields
\begin{align*}
    w = u + \pi_K (w - u) + F(\pi_K (w - u)).
\end{align*}
This completely proves property (\ref{LS2}).

\medskip

We prove property (\ref{LS3}), i.e. the estimate for \(\nabla F\). This part is formally the same as in the closed Yamabe case, but the elliptic estimate used is not the classical Schauder estimate for a uniformly elliptic equation. Instead, one uses the fractional Schauder estimates, see \cite{NirembergProb,GonzPS}.
Let $\psi \in K$, by Schauder estimates, see \cite[Lemma 5.2]{GonzPS}
\begin{align*}
    \|\nabla F (\phi)[\psi] \|_{C^{2\gamma + \alpha}(M)} \leq C \| \nabla_{\mathcal{B}}^2 I_\gamma [u] [\nabla F (\phi)[\psi]]  \|_{C^{\alpha}(M)},
\end{align*}
where we used the fact $\nabla F (\phi)[\psi] \in K^\perp$ to drop the lower order term. Next, using (\ref{eq:LS_nablaF})
\begin{align*}
    \nabla_{\mathcal{B}}^2 I_\gamma [u] [\nabla F (\phi)[\psi]]  &= \nabla_{\mathcal{B}}^2 I_\gamma [u] [\pi_{K^\perp} \nabla \mathcal{N}^{-1} (\phi)[\psi]] \\
    &= \pi_{K^\perp} \nabla_{\mathcal{B}}^2 I_\gamma [u] [ \nabla \mathcal{N}^{-1} (\phi)[\psi]] \, ,
\end{align*}
where the second equality follows from $\nabla_{\mathcal{B}}^2 I_\gamma [u](\phi) \in K^\perp$ for all $\phi \in H^\gamma (M)$. So
\begin{align*}
    \|\nabla F (\phi)[\psi] \|_{C^{2\gamma +\alpha}(M)} \leq C \| \pi_{K^\perp} \nabla_{\mathcal{B}}^2 I_\gamma [u] [ \nabla \mathcal{N}^{-1} (\phi)[\psi]] \|_{C^{\alpha}(M)}.
\end{align*}
Next, we claim that
\begin{align}\label{eq:LS_claim}
    \| \pi_{K^\perp} \nabla_{\mathcal{B}}^2 I_\gamma [u] [ \nabla \mathcal{N}^{-1} (\phi)[\psi]] \|_{C^{\alpha}(M)} \leq \varepsilon \|\nabla \mathcal{N}^{-1} (\phi)[\psi]\|_{C^{2\gamma +\alpha}(M)}.
\end{align}
Using that $\phi \in \mathcal{U}$ has the decomposition (\ref{eq:LS_Phi})
\begin{align*}
    \phi = \pi_K (\mathcal{N}^{-1}(\phi) - u) + \pi_{K^\perp} \nabla_{\mathcal{B}} I_\gamma (\mathcal{N}^{-1} (\phi)),
\end{align*}
and differentiating in the direction of $\psi \in K$
\begin{align*}
    \psi = \pi_K (\nabla \mathcal{N}^{-1}(\phi) [\psi] ) + \pi_{K^\perp} \nabla_{\mathcal{B}}^2 I_\gamma (\mathcal{N}^{-1} (\phi))[\nabla \mathcal{N}^{-1} (\phi)[\psi]].
\end{align*}
Take $\pi_{K^\perp}$
\begin{align*}
    0 = \pi_{K^\perp} \nabla_{\mathcal{B}}^2 I_\gamma (\mathcal{N}^{-1} (\phi))[\nabla \mathcal{N}^{-1} (\phi)[\psi]].
\end{align*}
Using this in the left-hand side of (\ref{eq:LS_claim})
\begin{align*}
    \| \pi_{K^\perp} \nabla_{\mathcal{B}}^2 I_\gamma [u] [ \nabla \mathcal{N}^{-1} (\phi)[\psi]] \|_{C^{\alpha}(M)} =& \| \pi_{K^\perp} \left(\nabla_{\mathcal{B}}^2 I_\gamma [u] - \nabla_{\mathcal{B}}^2 I_\gamma [\mathcal{N}^{-1}(\phi)] \right)[ \nabla \mathcal{N}^{-1} (\phi)[\psi]] \|_{C^\alpha (M)} \\
    \leq&C \|\nabla_{\mathcal{B}}^2 I_\gamma [u] - \nabla_{\mathcal{B}}^2 I_\gamma [\mathcal{N}^{-1}(\phi)] \|_{ L(C^{2\gamma + \alpha},C^\alpha ) } \| \nabla \mathcal{N}^{-1} (\phi)[\psi] \|_{C^{2\gamma + \alpha}(M)}\\
    \leq& \omega (\|u - \mathcal{N}^{-1}(\phi) \|_{C^{2\gamma + \alpha}(M)}) \| \nabla \mathcal{N}^{-1} (\phi)[\psi] \|_{C^{2\gamma + \alpha}(M)}\\
    \leq& \omega (\| \phi \|_{C^{ \alpha}(M)}) \| \nabla \mathcal{N}^{-1} (\phi)[\psi] \|_{C^{2\gamma + \alpha}(M)}\\
    \leq&  \tilde{\omega} (\| \phi \|_{H^\gamma (M)}) \| \nabla \mathcal{N}^{-1} (\phi)[\psi] \|_{C^{2\gamma + \alpha}(M)},
\end{align*}
where the second inequality comes from the continuity of the Hessian, the third one from the continuity of $\mathcal{N}^{-1}$ from $C^{ \alpha}(M)$ to $C^{2\gamma + \alpha}(M)$, and the last one from the equivalence of norms in $K$.

With the claim (\ref{eq:LS_claim}) proved, we use $\nabla \mathcal{N}^{-1} (\phi)[\psi] = \psi + \nabla F(\phi)[\psi]$, to get
\begin{align*}
    \|\nabla F(\phi)[\psi]\|_{C^{2\gamma + \alpha}(M)} \leq \varepsilon (\|\psi\|_{C^{\alpha}(M)} + \|\nabla F(\phi)[\psi]\|_{C^{2\gamma + \alpha}(M)}).
\end{align*}
Absorbing the second term, we get the desired conclusion (\ref{LS3}).

\end{proof}

\section{Compactness of minimizing sequences} \label{appendix:B}
In this appendix we show with some level of detail that every normalized minimizing sequence for the extension functional converges strongly, up to subsequence, to a unit norm minimizer of the problem. Using the trace operator, we can translate this into compactness of minimizing sequences of the non-local functional. This allows us to turn the local estimates shown in Proposition \ref{Proposition:localstab} into global estimates, which is our main Theorem \ref{thm:mainv1} of this manuscript. The proof relies on a concentration compactness principle from P. L. Lions adapted to weighted spaces. See \cite{Lions1,Lions2} for the original method and \cite{Valdinocci} for a weighted version. We shall mention also \cite{NeumayerCC,Struwe}.

There are several approaches to proving the compactness of minimizing sequences. While \cite{GonzPS} studies the blow-up behavior of Palais--Smale sequences, our proof is based on the more analytical method concentration--compactness.

To apply the concentration--compactness principle we first need a Sobolev trace inequality for measures, analogous to Lemma 1.2 in \cite{Lions1}. For this purpose we use the following result of Jin and Xiong \cite{JinXiong}.

\begin{proposition}\label{ineq:trace_mdm}
    For any $\varepsilon>0$ there exists $C_\varepsilon >0$ such that
    \begin{align*}
        \left( \int_M |u|^{2^*} d\sigma_{\hat{h}}   \right)^{2/2^*} \leq (\bar{S}(n,\gamma) + \varepsilon) \int_X y^a |\nabla U|^2 dv_{\bar{g}^* } + C_\varepsilon \int_X y^a U^2 dv_{\bar{g}^* } \, .
    \end{align*}
\end{proposition}

We aim to use this almost Sobolev trace inequality to derive one for measures and apply concentration compactness arguments, see for instance Lemma 1.2 in \cite{Lions1}, Lemma 2.3 \cite{Lions2} or Proposition 3.2.1 in \cite{Valdinocci}. In particular, we follow the approach of \cite{Valdinocci}.

\begin{lemma}[Concentration compactness]\label{Lemma:CC}
    Let $(U_k)_{k\in \mathbb{N}}$ be a bounded tight sequence in $W^{1,2}(X,y^a)$ such that $U_k$ converges weakly to $U$ in $W^{1,2}(X,y^a)$. Let $\mu$, $\nu$ be two nonnegative measures on $X$ and $M$ respectively and such that
    \begin{equation*}
        \lim_{k\to \infty} y^a |\nabla U_k|^2 = \mu \, ,
    \end{equation*}
    and
    \begin{align*}
        &\lim_{k\to \infty} |U_k (-,0)|^{2^*} = \nu .
    \end{align*}
    Then, there exists an at most countable set $J$ and three families $\{x_j\}_{j\in J} \subset M, \{\nu_j\}_{j\in J}, \{\mu_j\}_{j\in J} \in \mathbb{R}_+$ such that
    \begin{enumerate}
        \item $\nu = |U(-,0)|^{2^*} + \sum_{j\in J} \nu_j \delta_{x_j}$,
        \item $\mu \geq y^a |\nabla U|^2 + \sum_{j\in J} \mu_j \delta_{ ( x_j , 0 ) }$,
        \item $\mu_j \geq \bar{S}(n,\gamma)^{-1} \nu_j^{2/2^{*}}$ for all $j\in J$.
    \end{enumerate}
\end{lemma}
\begin{proof}
    Let $W_k:=U_k-U \rightharpoonup 0 $ in $W^{1,2}(X,y^a)$. By the compact embedding $W^{1,2}(X,y^a)\hookrightarrow L^2(X,y^a)$ we have $W_k\to 0 $ strongly in $L^2(X,y^a)$. After passing to a subsequence, there exist nonnegative bounded
    Radon measures $\widetilde{\mu}$ on $ X$ and
    $\widetilde{\nu}$ on $M$ such that
    \begin{align}
        y^a|\nabla W_k|^2\,dv_{\bar{g}^*}
        &\rightharpoonup^*
        \widetilde{\mu} \, 
        ,
        \label{eq:measures1}\\
        |W_k(-,0)|^{2^*}\,d\sigma_{\hat{h}}
        &\rightharpoonup^*
        \widetilde{\nu}\, .
        \label{eq:measures2}
    \end{align}

    We first show that the defect measures
    $\widetilde{\mu}$ and $\widetilde{\nu}$ satisfy a Sobolev trace
    inequality to then apply concentration compactness from \cite{Lions1,Lions2}. Let $\varphi\in C^\infty(X)$. Applying Proposition \ref{ineq:trace_mdm} to $\varphi W_k$, we get
    \begin{align*}
        \left(
        \int_M |\varphi W_k|^{2^*}\,d\sigma_{\hat{h}}
        \right)^{2/2^*}
        \leq&
        \big(\bar{S}(n,\gamma)+\varepsilon\big)
        \int_X y^a|\nabla(\varphi W_k)|^2\,dv_{\bar{g}^*}
        +
        C_\varepsilon
        \int_X y^a\varphi^2W_k^2\,dv_{\bar{g}^*}.
    \end{align*}
    Since $W_k\to0$ strongly in $L^2(X,y^a)$
    \begin{align}
        \int_X y^a\varphi^2W_k^2\,dv_{\bar{g}^*}=o(1).
        \label{eq:zeroorderCC}
    \end{align}
    On the other hand
    \begin{align*}
        \int_X y^a|\nabla(\varphi W_k)|^2\,dv_{\bar{g}^*}
        =&
        \int_X y^a\varphi^2|\nabla W_k|^2\,dv_{\bar{g}^*}
        +
        \int_X y^aW_k^2|\nabla\varphi|^2\,dv_{\bar{g}^*}
        +
        2\int_X y^a\varphi W_k
        \langle\nabla\varphi,\nabla W_k\rangle_g\,dv_{\bar{g}^*}.
    \end{align*}
    The second term converges to zero because
    $W_k\to0$ in $L^2(X,y^a)$. For the last term, the
    Cauchy-Schwarz inequality gives
    \begin{align*}
        &\left|
        \int_X y^a\varphi W_k
        \langle\nabla\varphi,\nabla W_k\rangle_g\,dv_{\bar{g}^*}
        \right|
        \leq
        \left(
        \int_X y^aW_k^2|\nabla\varphi|^2\,dv_{\bar{g}^*}
        \right)^{1/2}
        \left(
        \int_X y^a\varphi^2|\nabla W_k|^2\,dv_{\bar{g}^*}
        \right)^{1/2}.
    \end{align*}
    The first factor converges to zero, while the second one is
    uniformly bounded since $(W_k)_k$ is bounded in
    $W^{1,2}(X,y^a)$. Therefore we end up with
    \begin{align}
        \int_X y^a|\nabla(\varphi W_k)|^2\,dv_{\bar{g}^*}
        =
        \int_X y^a\varphi^2|\nabla W_k|^2\,dv_{\bar{g}^*}
        +o(1).
        \label{eq:gradientCC}
    \end{align}
    Combining \eqref{eq:zeroorderCC} and \eqref{eq:gradientCC}, we get
    \begin{align*}
        \left(
        \int_M |\varphi|^{2^*}|W_k|^{2^*}\,d\sigma_{\hat{h}}
        \right)^{2/2^*}
        \leq
        \big(\bar{S}(n,\gamma)+\varepsilon\big)
        \int_X y^a\varphi^2|\nabla W_k|^2\,dv_{\bar{g}^*}
        +o(1).
    \end{align*}
    Passing to the limit using \eqref{eq:measures1} and
    \eqref{eq:measures2}, we obtain
    \begin{align*}
        \left(
        \int_M |\varphi|^{2^*}\,d\widetilde{\nu}
        \right)^{2/2^*}
        \leq
        \big(\bar{S}(n,\gamma)+\varepsilon\big)
        \int_{X}\varphi^2\,d\widetilde{\mu}.
    \end{align*}
    Since this inequality holds for every $\varepsilon>0$, letting
    $\varepsilon\to0$ gives
    \begin{align}
        \left(
        \int_M |\varphi|^{2^*}\,d\widetilde{\nu}
        \right)^{2/2^*}
        \leq
        \bar{S}(n,\gamma)
        \int_{ X}\varphi^2\,d\widetilde{\mu}.
        \label{eq:measuretraceCC}
    \end{align}

    Applying the measure-theoretic concentration-compactness of
    \cite{Lions1,Lions2} to
    $\widetilde{\mu}$ and $\widetilde{\nu}$, there exists an at most countable set
    $J$, a family of points $(x_j)_{j\in J}\subset M$, and families
    $(\nu_j)_{j\in J}$ and $(\mu_j)_{j\in J}$ of positive numbers such
    that
    \begin{align}
        \widetilde{\nu}
        &=
        \sum_{j\in J}\nu_j\delta_{x_j},
        \label{eq:atomicdefectnu}\\
        \widetilde{\mu}
        &\geq
        \sum_{j\in J}\mu_j\delta_{(x_j,0)},
        \label{eq:atomicdefectmu}\\
        \mu_j
        &\geq
        \bar{S}(n,\gamma)^{-1}\nu_j^{2/2^*} 
        \quad \forall  j\in J.
        \label{eq:atomicrelationCC}
    \end{align}

    It remains to recover the measures $\mu$ and $\nu$ associated
    with the original sequence $(U_k)_{k\in \mathbb{N}}$. After passing to a further subsequence, we may assume that
    $U_k\to U$ almost everywhere on $M$. By the Brezis--Lieb lemma,
    for every $\varphi\in C^\infty(X)$,
    \begin{align*}
        \int_M |\varphi U_k|^{2^*}\,d\sigma_{\hat{h}}
        =
        \int_M |\varphi U|^{2^*}\,d\sigma_{\hat{h}}
        +
        \int_M |\varphi W_k|^{2^*}\,d\sigma_{\hat{h}}
        +o(1).
    \end{align*}
    Passing to the limit, we obtain
    \begin{align*}
        \int_M |\varphi|^{2^*}\,d\nu
        =
        \int_M |\varphi|^{2^*}|U|^{2^*}\,d\sigma_{\hat{h}}
        +
        \int_M |\varphi|^{2^*}\,d\widetilde{\nu}.
    \end{align*}
    Consequently,
    \begin{align*}
        \nu
        =
        |U(-,0)|^{2^*}\,d\sigma_{\hat{h}}+\widetilde{\nu}.
    \end{align*}
    Using \eqref{eq:atomicdefectnu}, we conclude that
    \begin{align*}
        \nu
        =
        |U(-,0)|^{2^*}\,d\sigma_{\hat{h}}
        +
        \sum_{j\in J}\nu_j\delta_{x_j}.
    \end{align*}

    Similarly, since $U_k=U+W_k$, for every
    $\varphi\in C^\infty( X)$ we have
    \begin{align*}
        \int_X y^a\varphi^2|\nabla U_k|^2\,dv_{\bar{g}^*}
        =&
        \int_X y^a\varphi^2|\nabla U|^2\,dv_{\bar{g}^*}
        +
        \int_X y^a\varphi^2|\nabla W_k|^2\,dv_{\bar{g}^*}
        +
        2\int_X y^a\varphi^2
        \langle\nabla U,\nabla W_k\rangle_g\,dv_{\bar{g}^*}.
    \end{align*}
    Since $W_k\rightharpoonup 0$ in $W^{1,2}(X,y^a)$, the last term converges to zero. Passing to the limit, we get
    \begin{align*}
        \int_{X}\varphi^2\,d\mu
        =
        \int_X y^a\varphi^2|\nabla U|^2\,dv_{\bar{g}^*}
        +
        \int_{ X}\varphi^2\,d\widetilde{\mu}.
    \end{align*}
    Therefore
    \begin{align*}
        \mu
        =
        y^a|\nabla U|^2\,dv_{\bar{g}^*}+\widetilde{\mu}.
    \end{align*}
    Finally, using \eqref{eq:atomicdefectmu}, we obtain
    \begin{align*}
        \mu
        \geq
        y^a|\nabla U|^2\,dv_{\bar{g}^*}
        +
        \sum_{j\in J}\mu_j\delta_{(x_j,0)}.
    \end{align*}
    Together with \eqref{eq:atomicrelationCC}, this concludes the
    proof.
\end{proof}

\begin{lemma}[Compactness for minimizing sequences for $I_\gamma^*$]\label{lemma:compacntess_ext}
    Let $(M^n , [\hat{h}])$ be the conformal infinity of an asymptotically hyperbolic manifold $(X^{n+1},g^+ )$ such that $\Lambda_{\gamma} (M,[\hat{h}]) < \Lambda_{\gamma} (\mathbb{S}^n ,[h_{0}])$. Let $(U_k)_{k\in \mathbb{N}} $ be a normalized minimizing sequence, i.e. $I_{\gamma}^* [U_k,\bar{g}^*] \to \Lambda_{\gamma} (M,[\hat{h}])$. Then, up to subsequence, $(U_k)_{k\in \mathbb{N}}$ converges strongly in $W^{1,2}(X,y^a)$ to some $U \in \mathcal{M}_1^*$.
\end{lemma}

\begin{proof}
    First, we may assume that $\Lambda_\gamma (M,[\hat{h}]) \geq 0$ since the negative case follows easily from the direct method. Let $(U_k)_{k\in \mathbb{N}} $ be a normalized minimizing sequence, that is
    \begin{align*}
        I_{\gamma}^* [U_k,\bar{g}^*]
        \longrightarrow
        \Lambda_{\gamma} (M,[\hat{h}]).
    \end{align*}
    Since $\|U_k\|_{2^*}=1$ for every $k\in\mathbb{N}$, this means that
    \begin{align*}
        \lim_{k\to\infty}
        \left(
        d_\gamma^*
        \int_X y^a |\nabla U_k|^2\,dv_{\bar{g}^*}
        +
        \int_M Q_\gamma^{\hat{h}} U_k^2\,d\sigma_{\hat{h}}
        \right)
        =
        \Lambda_\gamma(M,[\hat h]).
    \end{align*}
    Since $(U_k)_k$ is bounded in $W^{1,2}(X,y^a)$, after passing to a
    subsequence, there exists $U\in W^{1,2}(X,y^a)$ such that
    \begin{align*}
        U_k&\rightharpoonup U
        \quad \text{in }W^{1,2}(X,y^a) \text{ and } L^{2^*}(M),\\
        U_k&\to U
        \quad\text{in }L^2(X,y^a),\\
        U_k&\to U
        \quad\text{in }L^2(M).
    \end{align*}
    In particular,
    \begin{align}
        \int_M Q_\gamma^{\hat{h}} U_k^2\,d\sigma_{\hat{h}}
        \longrightarrow
        \int_M Q_\gamma^{\hat{h}} U^2\,d\sigma_{\hat{h}}.
        \label{eq:Qcompactness}
    \end{align}

    After passing to a further subsequence, there exist nonnegative
    bounded Radon measures $\mu$ on $X$ and $\nu$ on $M$ such
    that
    \begin{align*}
        y^a|\nabla U_k|^2\,dv_{\bar{g}^*}
        &\rightharpoonup^*\mu \, ,\\
        |U_k(-,0)|^{2^*}\,d\sigma_{\hat{h}}
        &\rightharpoonup^*\nu \, .
    \end{align*}
    Applying Lemma \ref{Lemma:CC}, there exist an at most countable set
    $J$, points $(x_j)_{j\in J}\subset M$, and positive numbers
    $(\nu_j)_{j\in J}$ and $(\mu_j)_{j\in J}$ such that
    \begin{align}
        \nu
        &=
        |U(-,0)|^{2^*}\,d\sigma_{\hat{h}}
        +
        \sum_{j\in J}\nu_j\delta_{x_j},
        \label{eq:measurelimit1}\\
        \mu
        &\geq
        y^a|\nabla U|^2\,dv_{\bar{g}^*}
        +
        \sum_{j\in J}\mu_j\delta_{(x_j,0)},
        \label{eq:measurelimit2}\\
        \mu_j
        &\geq
        \bar{S}(n,\gamma)^{-1}\nu_j^{2/2^*}
        \quad \forall j\in J.
        \label{eq:measurelimit3}
    \end{align}
    Since $d_\gamma^*\bar{S}(n,\gamma)^{-1}
        =
        \Lambda_\gamma(\mathbb{S}^n,[h_0])$ it follows that
    \begin{align}
        d_\gamma^*\mu_j
        \geq
        \Lambda_\gamma(\mathbb{S}^n,[h_0])
        \nu_j^{2/2^*} 
        \quad \forall j\in J.
        \label{eq:sphericalatom}
    \end{align}

    We claim that $\|U\|_{2^*}=1$. Let $t:= \|U\|_{2^*}^{2^*} \in [0,1]$. Testing the measure convergence defining $\nu$ with the constant
    function $1$, and using \eqref{eq:measurelimit1}, we obtain
    \begin{align}
        1
        =
        \lim_{k\to\infty}
        \int_M|U_k|^{2^*}\,d\sigma_{\hat{h}}
        =
        t+\sum_{j\in J}\nu_j.
        \label{eq:massdecomposition}
    \end{align}

    Similarly, testing the convergence of the energy measures with the
    constant function $1$ and using \eqref{eq:measurelimit2}, we get
    \begin{align*}
        \lim_{k\to\infty}
        \int_X y^a|\nabla U_k|^2\,dv_{\bar{g}^*}
        \geq
        \int_X y^a|\nabla U|^2\,dv_{\bar{g}^*}
        +
        \sum_{j\in J}\mu_j.
    \end{align*}
    Therefore, by \eqref{eq:Qcompactness} and
    \eqref{eq:sphericalatom},
    \begin{align*}
        \Lambda_\gamma(M,[\hat h])
        \geq&
        d_\gamma^*
        \int_X y^a|\nabla U|^2\,dv_{\bar{g}^*}
        +
        \int_M Q_\gamma^{\hat{h}} U^2\,d\sigma_{\hat{h}}
        +
        \sum_{j\in J}d_\gamma^*\mu_j\\
        \geq&
        d_\gamma^*
        \int_X y^a|\nabla U|^2\,dv_{\bar{g}^*}
        +
        \int_M Q_\gamma^{\hat{h}} U^2\,d\sigma_{\hat{h}}
        +
        \Lambda_\gamma(\mathbb{S}^n,[h_0])
        \sum_{j\in J}\nu_j^{2/2^*}.
    \end{align*}
    By the definition of $\Lambda_\gamma(M,[\hat h])$
    \begin{align*}
        d_\gamma^*
        \int_X y^a|\nabla U|^2\,dv_{\bar{g}^*}
        +
        \int_M Q_\gamma^{\hat{h}} U^2\,d\sigma_{\hat{h}}
        \geq
        \Lambda_\gamma(M,[\hat h])
        t^{2/2^*}.
    \end{align*}
    Hence
    \begin{align*}
        \Lambda_\gamma(M,[\hat h])
        \geq&
        \Lambda_\gamma(M,[\hat h])
        t^{2/2^*}
        +
        \Lambda_\gamma(\mathbb{S}^n,[h_0])
        \sum_{j\in J}\nu_j^{2/2^*}.
    \end{align*}
    Note that if $\Lambda_\gamma (M,[\hat{h}]) = 0$, this already implies no concentration is possible, so from now on we may assume $\Lambda_\gamma (M,[\hat{h}]) > 0$. Since $2/2^*\in(0,1)$, the map $s\longmapsto s^{2/2^*}$ is subadditive. Thus
    \begin{align*}
        \sum_{j\in J}\nu_j^{2/2^*}
        \geq
        \left(
        \sum_{j\in J}\nu_j
        \right)^{2/2^*}.
    \end{align*}
    Using \eqref{eq:massdecomposition}, we obtain
    \begin{align*}
        \Lambda_\gamma(M,[\hat h])
        \geq&
        \Lambda_\gamma(M,[\hat h])
        t^{2/2^*}
        +
        \Lambda_\gamma(\mathbb{S}^n,[h_0])
        (1-t)^{2/2^*}\\
        \geq&
        \Lambda_\gamma(M,[\hat h])
        \left(
        t^{2/2^*}
        +
        (1-t)^{2/2^*}
        \right).
    \end{align*}
    Since $\Lambda_\gamma(M,[\hat h])>0$ and $t^{2/2^*}+(1-t)^{2/2^*}>1$ for every $t\in (0,1)$, it follows that $t\in\{0,1\}$. If $t=0$, then the previous
    inequality gives
    \begin{align*}
        \Lambda_\gamma(M,[\hat h])
        \geq
        \Lambda_\gamma(\mathbb{S}^n,[h_0]),
    \end{align*}
    contradicting the hypothesis. Therefore $t=1$, and hence $\|U\|_{2^*}=1$. In particular, \eqref{eq:massdecomposition} implies that
    \begin{align*}
        \sum_{j\in J}\nu_j=0,
    \end{align*}
    so there is no concentration. Now, using this latter fact, we show strong convergence in $L^{2^*}(M)$. Indeed, $U_k\rightharpoonup U$ in $L^{2^*}(M)$ and $\|U_k\|_{2^*}=\|U\|_{2^*}=1$ the uniform convexity of $L^{2^*}(M)$ implies
    \begin{align}
        U_k\to U
        \quad\text{strongly in }L^{2^*}(M).
        \label{eq:strongLcritical}
    \end{align}

    We now show strong convergence in $W^{1,2}(X,y^a)$. Since
    $\|U\|_{2^*}=1$, the definition of the fractional Yamabe constant
    gives
    \begin{align}
        \Lambda_\gamma(M,[\hat h])
        \leq
        d_\gamma^*
        \int_X y^a|\nabla U|^2\,dv_{\bar{g}^*}
        +
        \int_M Q_\gamma^{\hat{h}} U^2\,d\sigma_{\hat{h}}.
        \label{eq:lowerboundlimit}
    \end{align}
    On the other hand, by weak lower semicontinuity of the weighted
    gradient norm and \eqref{eq:Qcompactness},
    \begin{align*}
        &d_\gamma^*
        \int_X y^a|\nabla U|^2\,dv_{\bar{g}^*}
        +
        \int_M Q_\gamma^{\hat{h}} U^2\,d\sigma_{\hat{h}}\\
        &\qquad\leq
        \liminf_{k\to\infty}
        \left(
        d_\gamma^*
        \int_X y^a|\nabla U_k|^2\,dv_{\bar{g}^*}
        +
        \int_M Q_\gamma^{\hat{h}} U_k^2\,d\sigma_{\hat{h}}
        \right)\\
        &\qquad=
        \Lambda_\gamma(M,[\hat h]).
    \end{align*}
    Together with \eqref{eq:lowerboundlimit}, this yields
    \begin{align}
        d_\gamma^*
        \int_X y^a|\nabla U|^2\,dv_{\bar{g}^*}
        +
        \int_M Q_\gamma^{\hat{h}} U^2\,d\sigma_{\hat{h}}
        =
        \Lambda_\gamma(M,[\hat h]).
        \label{eq:limitminimizer}
    \end{align}
    Thus, $U\in\mathcal{M}_1^*$. Finally, using the minimizing property, \eqref{eq:Qcompactness},
    and \eqref{eq:limitminimizer}, we obtain
    \begin{align*}
        d_\gamma^*
        \int_X y^a|\nabla U_k|^2\,dv_{\bar{g}^*}
        =
        I_\gamma^*[U_k, \bar{g}^*]
        -
        \int_M Q_\gamma^{\hat{h}} U_k^2\,d\sigma_{\hat{h}}
        \longrightarrow&
        \Lambda_\gamma(M,[\hat h])
        -
        \int_M Q_\gamma^{\hat{h}} U^2\,d\sigma_{\hat{h}}\\
        =&
        d_\gamma^*
        \int_X y^a|\nabla U|^2\,dv_{\bar{g}^*}.
    \end{align*}
    Therefore, we have convergence of the norms $\|\nabla U_k\|_{L^2(X,y^a)}
        \longrightarrow
        \|\nabla U\|_{L^2(X,y^a)}$. Since $\nabla U_k\rightharpoonup\nabla U$ in $L^2(X,y^a)$, the uniform convexity of the weighted $L^2$ space gives $\nabla U_k\to\nabla U$ strongly in $L^2(X,y^a)$.
    Together with the already known convergence $U_k \to U$ in $L^2 (X,y^a)$, we proved that $U_k\to U$ strongly in $W^{1,2}(X,y^a)$. Finally, since $U$ is a minimizer of $I_\gamma^*$, the regularity and maximum principle results of \cite{GonzQing} imply that $U$ has the required regularity and is strictly positive. Thus its trace $u=T(U)$ is a positive admissible conformal factor, and hence $u\in\mathcal{M}_1$.
\end{proof}

\begin{lemma}[Compactness of minimizing sequences for $I_\gamma$]
    Let $(M^n , [\hat{h}])$ be the conformal infinity of an asymptotically hyperbolic manifold $(X^{n+1},g^+ )$ such that $\Lambda_{\gamma} (M,[\hat{h}]) < \Lambda_{\gamma} (\mathbb{S}^n ,[h_{0}])$. Let $(w_k)_{k\in \mathbb{N}} \subset \mathcal{B}$ be a minimizing sequence, i.e. $I_{\gamma} [w_k,\hat{h}] \to \Lambda_{\gamma} (M,[\hat{h}])$. Then, up to subsequence, $(w_k)_{k\in \mathbb{N}}$ converges strongly in $H^{\gamma}(M)$ to some $w_* \in \mathcal{M}_1$.
\end{lemma}
\begin{proof}
    Let $(w_k)_{k\in \mathbb{N}} \subset \mathcal{B}$ be a minimizing sequence for $I_\gamma[-,\hat{h}]$. Let $U_k:=U_{w_k}$ denote the weighted harmonic extension of $w_k$. By the equivalence between the fractional and extension formulations, and since $\|w_k\|_{L^{2^*}(M)}=1$, we have
    \begin{align*}
        I_\gamma[w_k,\hat{h}]
        &= \int_M w_k P_\gamma^{\hat{h}}(w_k)\,d\sigma_{\hat{h}}\\
        &= d_\gamma^*\int_X y^a |\nabla U_k|^2\,dv_{\bar{g}^*}
        + \int_M Q_\gamma^{\hat{h}}w_k^2\,d\sigma_{\hat{h}} = I_\gamma^*[U_k,\bar{g}^*]
        \longrightarrow \Lambda_\gamma(M,[\hat{h}]).
    \end{align*}
    Thus $(U_k)_{k\in\mathbb N}$ is a normalized minimizing sequence for the extension problem. By Lemma \ref{lemma:compacntess_ext}, after passing to a subsequence, there exists $U\in\mathcal{M}_1^*$ such that $U_k\to U $ in $W^{1,2}(X,y^a)$. The continuity of the trace operator then yields
    \begin{align*}
        w_k=T(U_k)\to T(U)=:w
        \qquad\text{in }H^\gamma(M).
    \end{align*}
    Since $U\in\mathcal{M}_1^*$, the equivalence of the two minimization problems implies that $w\in\mathcal{M}_1$. Hence every minimizing sequence admits a subsequence converging strongly in $H^\gamma(M)$ to a normalized minimizer.
\end{proof}

\end{document}